\documentclass[11pt]{article}
\usepackage[T1]{fontenc}

\usepackage[a4paper,margin=1in]{geometry}
\usepackage{indentfirst}

\usepackage{amsmath,amsthm,amsfonts,amssymb,bm,wasysym}
\usepackage[usenames]{color}
\usepackage{hyperref}
\usepackage[shortlabels]{enumitem}
\usepackage{graphicx}
\usepackage[singlelinecheck=false]{caption}
\usepackage{tikz-cd}
\usepackage{faktor}

\numberwithin{equation}{section}

\theoremstyle{definition}

\newtheorem{ateo}{Theorem}

\newtheorem{lema}{Lemma}[section]
\newtheorem{prop}{Proposition}[section]
\newtheorem{deft}{Definition}[section]
\newtheorem{cor}{Corollary}[section]

\newtheorem{rmk}{Remark}[section]

\DeclareMathOperator\PSL{PSL}
\DeclareMathOperator\Per{Per}
\DeclareMathOperator\Cov{Cov}

\newcommand{\C}{\mathbb{C}}
\newcommand{\D}{\mathbb{D}}

\newcommand{\R}{\mathbb{R}}
\newcommand{\s}{\mathbb{S}}
\newcommand{\Z}{\mathbb{Z}}

\allowdisplaybreaks

\title{Matings between compositions of rational maps and free products of finite cyclic groups}
\author{Shaun Bullett, Luna Lomonaco, Miguel Ratis Laude}
\date{2026}

\begin{document}

\maketitle

\begin{abstract}
    Given a pair of rational maps $(f, g)$, of degrees $p$ and $q$, each with a parabolic fixed point having a fully invariant simply-connected basin of attraction, we construct an algebraic correspondence $F$ on the Riemann sphere, of bidegree $(pq, pq)$, realizing a mating between the two compositions $g\circ f$ and $f\circ g$ of the maps, and the parabolic faithful discrete representation of the free product of cyclic groups of orders $p+1$ and $q+1$. We also show that $F$ is the composition of a pair of deleted covering correspondences of rational maps which are conjugated to polynomials of degrees $p+1$ and $q+1$. We generalize our method to construct matings between compositions of pairs of polynomials and (non-parabolic) faithful Kleinian representations of the same group, now with connected regular set. As far as we are aware, these matings between pairs of maps and groups are the first examples that are not time-reversible (that is, they are not conjugate to their own inverses).
\end{abstract}

\section{Introduction}\label{sec:intro}

There are two classical ways of combining conformal dynamical systems. On the group side, Klein's combination theorem \cite{maskit-1965} builds a single Kleinian group from two groups acting on complementary fundamental domains, and Bers' simultaneous uniformization \cite{bers-1960} welds two conformal structures into one quasiFuchsian group; on the map side, Douady-Hubbard matings \cite{douady-1983} glue two filled Julia sets into a single rational map. A \textit{mating between a rational map and a Kleinian group} is the hybrid operation: one dynamical system behaving, on complementary invariant subsets of the Riemann sphere, as a map on a filled Julia set and as a group on its ordinary set. No single-valued map can do both jobs at once, but an \textit{algebraic correspondence} --- a multivalued map $z \mapsto w$ defined by a polynomial relation $P(z,w)=0$ --- can, as Fatou already suspected a century ago \cite{fatou-1929}. The first such matings were constructed by the first author and C.~Penrose \cite{bullett-penrose-1994}, for the modular group and quadratic polynomials, and by now there is a substantial theory (see Subsection~\ref{subsec:background}).

Matings must negotiate an asymmetry between their two ingredients: the arrow of time. A group is time-symmetric --- every generator comes with its inverse --- while a rational map of degree at least two is irreversible. A correspondence $F$ always has an inverse correspondence $F^{-1}$. In every mating constructed so far, three features are common: each mates a \textit{single} map with a group; the `limit set' consists of two copies of \textit{one and the same} filled Julia set; and the correspondence is \textit{time-reversible}, that is, $F^{-1}$ is conjugate to $F$. This is no accident. As observed in \cite[\S 5.3]{lomonaco-mukherjee-2026}, most correspondences realizing matings can be written in the form $J \circ \mathrm{Cov}_0^Q$, where
\[
\mathrm{Cov}_0^Q \colon\quad z \mapsto w \iff \frac{Q(w)-Q(z)}{w-z}=0
\]
is the \textit{deleted covering correspondence} of a holomorphic map $Q$ onto a quotient surface and $J$ is an involution of the surface; and a correspondence of this form is automatically conjugate to its own inverse bby the involution $J$.

In this paper we construct matings that break with all three features at once. We mate a \textit{pair} $(f,g)$ of rational maps, of degrees $p$ and $q$, with faithful discrete representations of the free product $C_{p+1} * C_{q+1}$ of finite cyclic groups ($p,q \geq 2$); the limit set of the resulting bidegree $(pq,pq)$ correspondence consists of copies of the filled Julia sets of the \textit{two} compositions $g \circ f$ and $f \circ g$, which are in general genuinely different (see Remark~\ref{rmk.distinct_julia_sets}); and the correspondences we obtain are of the more general form $\mathrm{Cov}_0^P \circ \mathrm{Cov}_0^Q$, with a commutator involved when time is reversed. In particular, they need not be time-reversible: we show that the mating cannot be conjugate to its own inverse unless the two compositions are conformally conjugate (Proposition~\ref{main.time-reversibility-condition}), and we exhibit explicit non-reversible examples. The failure of such reversibility reflects the non-commutativity of the pair: the correspondence carries copies of the filled Julia sets of both compositions, $K(g\circ f)$ and $K(f\circ g)$, on a single sphere --- with branches of $F$ and $F^{-1}$ acting on them as $g\circ f$ and $f\circ g$ respectively --- and it loses the symmetry under time revarsal as soon as the two compositions differ as conformal conjugacy classes.

\subsection{Statement of results}\label{subsec:results}

Our first results concern the parabolic side of the story. In Section~\ref{sec.preliminaries} we introduce the class $\mathcal{P}_d$ of degree $d$ rational maps for which $\infty$ is a fixed point of multiplier $1$ with a unique attracting direction, and the subclass $\mathrm{Per}_1^d(1)$ of maps possessing a fully invariant connected immediate basin of attraction; the complement of that basin is a completely invariant full compact set, the \textit{filled Julia set} $K(\cdot)$. On the group side, $\Gamma_{p,q}$ denotes the representation of $C_{p+1} * C_{q+1}$ in $PSL(2,\R)$ for which the composition of the generators of orders $p+1$ and $q+1$ is parabolic: it is the orientation-preserving subgroup of the $(p+1,q+1,\infty)$-triangle group, and it is the unique faithful discrete representation of $C_{p+1} * C_{q+1}$ with limit set $\mathbb{R} \cup \{\infty\}$. For $q=1$ these are the Hecke groups $H_{p+1}$, and $(p,q)=(2,1)$ is the modular group; our matings, restricted to these cases, recover the notions of \cite{bullett-lomonaco-2019,bullett-lomonaco-2022,bullett-lomonaco-2024,bullett-lomonaco-lyubich-mukherjee-2024} (Remark~\ref{rmk.hecke}).

\begin{ateo}\label{main.parabolic_mating}
Let $p,q \geq 2$ be integers, $f \in \mathcal{P}_p$ and $g \in \mathcal{P}_q$ be parabolic maps such that $g \circ f \in \mathrm{Per}_1^{pq}(1)$ and $K(g \circ f)$ is connected. Then there exists a holomorphic correspondence $F$ on $\widehat{\C}$ that realizes the \textit{mating} between the pair $(f,g)$ and the group $\Gamma_{p,q}$.
\end{ateo}

Here \textit{mating} (Definition~\ref{deft.parabolic_mating}) means: the sphere splits into fully $F$-invariant disjoint sets $\Lambda \cup \Omega$, where $\Lambda = \Lambda_- \cup \Lambda_+$ with $\Lambda_- \cap \Lambda_+$ a single point; the $pq:1$ restriction $F|_{\Lambda_-}$ admits a pinched polynomial-like restriction hybrid conjugate to $g \circ f$, and $F^{-1}|_{\Lambda_+}$ admits one hybrid conjugate to $f \circ g$; and the action of $F$ on the open set $\Omega$ is conformally conjugate to the action of the elements $\sigma^n \rho^m$ of $\Gamma_{p,q}$ on the unit disk. Figure~\ref{fig:mating-asymmetric} shows the mating between a pair of quadratic maps and the group $\Gamma_{2,2} \cong C_3 * C_3$.

\begin{figure}
    \centering
    \includegraphics[width=0.7\linewidth]{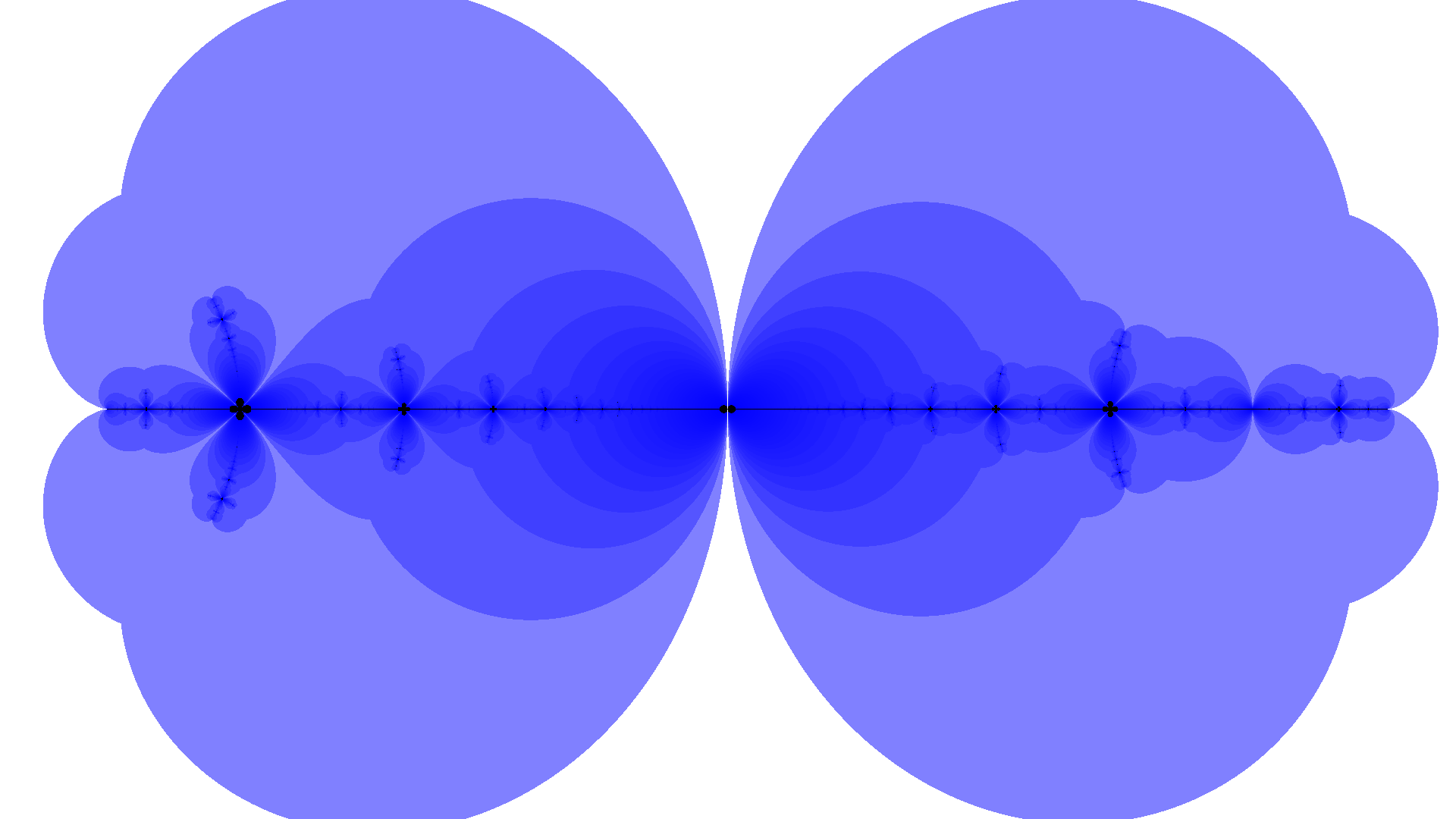}
    \caption{A mating between the a pair of quadratic maps and the group $\Gamma_{2,2}$.}
    \label{fig:mating-asymmetric}
\end{figure}

Our second result is a rigidity statement: it identifies exactly which correspondences can realize such matings.

\begin{ateo}\label{main.parabolic_cov}
Let $p,q \geq 2$ be integers, $f \in \mathcal{P}_p$, $g \in \mathcal{P}_q$ with $g \circ f \in \mathrm{Per}_1^{pq}(1)$. If $F$ is a correspondence realizing the mating between the pair $(f,g)$ and the group $\Gamma_{p,q}$, then there are rational maps $P,Q \colon \widehat{\C} \to \widehat{\C}$, of degrees $p+1$ and $q+1$ respectively, both
conjugate to polynomials, such that $F = \mathrm{Cov}_0^Q \circ \mathrm{Cov}_0^P$.
\end{ateo}

Our third and fourth results replace the parabolic representation $\Gamma_{p,q}$ by an arbitrary faithful discrete representation $\Gamma$ of $C_{p+1} * C_{q+1}$ in $PSL(2,\C)$ with totally disconnected limit set (all such representations form a single quasiconformal conjugacy class), and the parabolic maps by polynomials; the notion of mating is correspondingly weakened to \textit{orbit mating} (Definition~\ref{deft.hyperbolic_mating}), asking on the group side only for grand orbit equivalence, as in \cite{bullett-harvey-2000, ratis-laude-2025}.

\begin{ateo}\label{main.hyperbolic_mating}
Let $p,q \geq 2$ be integers, $f,g$ polynomials of degrees $p$ and $q$ respectively with $K(g \circ f)$ connected, and $\Gamma$ a faithful discrete representation of $C_{p+1} * C_{q+1}$ in $PSL(2,\C)$ with totally disconnected limit set. Then there exists a holomorphic correspondence $F$ on $\widehat{\C}$ realizing the orbit mating between the pair $(f,g)$ and the group $\Gamma$.
\end{ateo}

\begin{ateo}\label{main.hyperbolic_cov}
The orbit matings of Theorem~C likewise factor as $F = \mathrm{Cov}_0^Q \circ \mathrm{Cov}_0^P$ with $P,Q$ of degrees $p+1$ and $q+1$ conjugate to polynomials.
\end{ateo}

Finally, we make precise the assertion that these matings are new in kind. The mechanism behind the time-reversibility of all previous constructions can be exhibited in one line: if $Q$ is a rational map, $J$ is an involution of the sphere, and $F = J\circ \Cov_0^Q$, then $J$ conjugates $F$ and $F^{-1}$. Indeed, the defining relation of $\Cov_0^Q$ is symmetric in $z$ and $w$, so $\Cov_0^Q = (\Cov_0^Q)^{-1}$; hence $J\circ F\circ J = J\circ(J\circ \Cov_0^Q)\circ J = \Cov_0^Q\circ J = F^{-1}$. It follows that every mating of the classical form $J\circ \Cov_0^Q$ is time-reversible, and the involution J carries one limb of its limit set conformally onto the other; in particular, the two limbs are conformally equivalent. Conversely, reversibility constrains the limbs:

\begin{prop}\label{main.time-reversibility-condition}
Let $F$ be a mating between the pair $(f,g)$ and $\Gamma_{p,q}$ as in Theorem~A. If $F$ is conjugate to $F^{-1}$ --- by any M\"{o}bius map, not necessarily an involution --- then the compositions $g \circ f$ and $f \circ g$ are conformally conjugate rational maps.
\end{prop}

\begin{cor}
Let $F$ be a mating as in Theorem~A for which $g\circ f$ and $f\circ g$ are not conformally conjugate. Then $F$ is not conjugate to $F^{-1}$, and $F$ cannot be written in the form $J\circ \Cov_0^R$ for any involution $J$ and any rational map $R$.
\end{cor}

Thus, for instance, for the pair of Figure~\ref{fig:mating-asymmetric}, which is $f(z)=z^2/(z+1/2)$, and $g(z)=(2z^2-z-1/2)/(2z+1)$, the compositions are not hybrid equivalent and the mating is not conjugate to its own inverse (see the comment right after the proof of Proposition \ref{main.time-reversibility-condition}); whereas for $f(z)=g(z)=z^2/(z+1/2)$ (Figure~\ref{fig:matings}, left), the two compositions coincide and the mating is time-reversible. The criterion is not restricted to the correspondences constructed in this paper: any correspondence whose restricted forward and backward branches are conjugate to the two compositions of a pair of maps --- such as those of \cite{luo-mj-mukherjee-2026} --- is conjugate to its own inverse only if the two compositions are conformally equivalent, and is therefore generically not time-reversible. Whether the converse of Proposition 1.1 holds is a rigidity question, recorded as Question (1) in Subsection 1.4.

The dichotomy of Proposition 1.1 is governed by a question of independent classical interest: when are the two compositions $g\circ f$ and $f\circ g$ dynamically equivalent? For polynomials, this belongs to the decomposition theory founded by Ritt \cite{ritt-22,ritt-23}: the prime factors of a polynomial under composition are essentially unique, the admissible exchanges of adjacent factors are classified, and the commuting pairs are known (see also \cite{eremenko-1989, pakovich-2021}). For rational maps, no comparably complete theory exists, despite substantial partial results --- and our matings can be regarded as a dynamical probe of exactly this gap: the correspondence $F$ realizes both orderings of the pair $(f,g)$ simultaneously, together with intertwining both branches conjugate to $f$ and $g$ themselves, so that the difference between $\Lambda_-$ and $\Lambda_+$ makes the commutator of $f$ and $g$ \textit{visible} in a single dynamical plane. It is tempting to recall that the very first picture of what would later be known as the Mandelbrot set appeared in the work of Brooks and Matelski on discreteness of two-generator subsgroups of $\PSL(2, \C)$ --- an iteration on traces of words in a pair of M\"{o}bius transformations, tied to the commutator through Jorgensen's inequality. Firfty years later, pairs of maps and two-generator free products meet again on the sphere, with non-commutativity once more as the organizing theme.

\subsection{Background and prior constructions}\label{subsec:background}

The observation that algebraic correspondences can realize matings dates back to \cite{bullett-penrose-1994}, where certain correspondences in a one-parameter family were shown to mate the modular group $PSL(2,\Z)$ with quadratic polynomials having connected Julia set; the general framework of regular and limit sets for correspondences was developed in \cite{bullett-penrose-2001}, and a combination criterion for covering correspondences in \cite{bullett-2000}. Matings of Chebyshev-like maps with Hecke groups were constructed in \cite{bullett-freiberger-2003, bullett-freiberger-2005}, and a transcendental family in \cite{bullett-freiberger-2008}. In subsequent work by the two first authors \cite{bullett-lomonaco-2019, bullett-lomonaco-2022, bullett-lomonaco-2024} the picture became clearer: the Bullett--Penrose correspondences more appropriately realize matings between the modular group and the parabolic family $\mathrm{Per}_1(1)=\{\, z+1/z+A \;:\; A \in \mathbb{C} \,\}$, and this approach culminated in the theorem \cite{bullett-lomonaco-2024} that the `modular Mandelbrot set' is dynamically homeomorphic to the parabolic Mandelbrot set (itself homeomorphic to the Mandelbrot set by Petersen--Roesch \cite{petersen-roesch-2021}). A complete dynamical theory for this family --- B\"ottcher coordinates, landing of periodic geodesics, a sharpened Yoccoz inequality --- was developed in \cite{bullett-lomonaco-2022}, deformations of the modular group as quasifuchsian correspondences were studied in \cite{bullett-2010,bullett-curtis-2012}, and the parameter space program has recently been extended to Modular Multibrot sets and Belyi polynomials \cite{bullett-lobato-lomonaco-navarro-2025a,bullett-lobato-lomonaco-navarro-2025b,bullett-lobato-lomonaco-navarro-ratis-laude-2025}.

On the Kleinian side of the same school, the first author and W.~Harvey \cite{bullett-harvey-2000} applied the Douady--Hubbard theory of polynomial-like maps \cite{douady-hubbard-1985} to mate quadratic polynomials with representations of $C_2 * C_3$ having totally disconnected limit set; the third author \cite{ratis-laude-2025} generalized this to all degrees and proved analyticity and continuity of the mating map on parameter space; and P.~Ha\"issinsky and the first author \cite{bullett-haissinsky-2007} pinched the correspondences of \cite{bullett-harvey-2000} to reach the modular group itself.

Independently, motivated by problems of statistical physics and the theory of quadrature domains \cite{lee-makarov-2016}, M.~Lyubich, N.~Makarov, S.~Mukherjee and collaborators developed an anti-holomorphic counterpart: Schwarz reflection maps arising as conformal matings of reflection groups with anti-rational maps \cite{lee-lyubich-makarov-mukherjee-2021, lee-lyubich-makarov-mukherjee-2023, lee-lyubich-makarov-mukherjee-2025}, the David extension and welding machinery underlying them \cite{lyubich-merenkov-mukherjee-ntalampekos-2025}, realization theorems for antiholomorphic correspondences \cite{lyubich-mazor-mukherjee-2024a, lyubich-mazor-mukherjee-2024b}, and connections between limit sets of kissing reflection groups and Julia sets of critically fixed anti-rational maps \cite{lodge-luo-mukherjee-2022}; see \cite{lomonaco-mukherjee-2026} for a panorama.

Techniques from the holomorphic and anti-holomorphic lines were combined in \cite{bullett-lomonaco-lyubich-mukherjee-2024}, where the two first authors, together with M.~Lyubich and S.~Mukherjee, proved that all members of a broad family of parabolic rational maps are mateable with the Hecke groups, resolving the parabolic version of a conjecture of Bullett--Freiberger.

A third line, initiated by M.~Mj and S.~Mukherjee, combines groups and maps via orbit equivalence of Bowen--Series maps and David surgery: matings of polynomials with quasiFuchsian punctured-sphere groups \cite{mj-mukherjee-2023}, with genus zero orbifold groups via factor Bowen--Series maps and B-involutions \cite{mj-mukherjee-2025, luo-lyubich-mukherjee-2024}, extensions of Bers simultaneous uniformization to correspondences \cite{mj-mukherjee-2024}, Bers slices, polynomial loci and degeneration in spaces of correspondences \cite{luo-mj-mukherjee-2025}, and combination theorems on hyperelliptic surfaces \cite{mukherjee-viswanathan-2025}. Ergodic and arithmetic aspects of correspondences --- invariant measures, equidistribution, thermodynamic formalism, canonical heights --- have been studied by many authors; for the families closest to this paper see \cite{vanessa-2024, siqueira-smania-2017, siqueira-2023}, and see also \cite{bharali-sridharan-2016, dihn-kaufmann-wu-2020, hemmingsson-2024, ingram-2019}.

All of the mating constructions above share the three features described in the opening: a single map, a doubled filled Julia set, and time-reversibility.

\subsection{Origin of the program and relation to other work}\label{subsec:simultaneous}

The program carried out in this paper was announced in the ICM survey \cite{lomonaco-mukherjee-2026} (arXiv, November 2025), written by the second author jointly with S.~Mukherjee. Section~5.2 of that survey describes the correspondences constructed here --- $pq:pq$ correspondences of the form $\Cov_0^P \circ \Cov_0^Q$, with $P,Q$ polynomials of degrees $p+1$, $q+1$ sharing a common simple critical point, partitioning the sphere into an invariant open set $\Omega$ carrying the parabolic representation of $C_{p+1}*C_{q+1}$ and a one-point union $\Lambda_-\cup\Lambda_+$ of copies of $K(g\circ f)$ and $K(f\circ g)$ --- and cites for this the present work, listed there as reference [23]: ``\textit{Mating Rational Maps with $(p,q,\infty)$ Triangle Groups}, preprint, 2025''.

While this paper was being completed, the preprint \cite{luo-mj-mukherjee-2026} appeared (July 8, 2026), constructing --- for Fuchsian $(p,q,\infty)$-triangle groups with $p,q \geq 3$ --- correspondences of the form $\mathrm{Cov}_0^P \circ \mathrm{Cov}_0^Q$ combining the group with a pair of Blaschke products, via David surgery; a generalization to pairs of polynomials or parabolic rational maps is mentioned there as a remark \cite[Remark~4.9(3)]{luo-mj-mukherjee-2026}. None of the results of the present paper appear in \cite{luo-mj-mukherjee-2026}. See also Remark \ref{rmk.roots} for examples where the techniques present in this paper can also be used to obtain some of the correspondences constructed in \cite{luo-mj-mukherjee-2026}.


\subsection{Questions}\label{subsec:questions}

The results above suggest several directions.

\begin{enumerate}
\item \emph{Reversibility and Ritt theory.} By Proposition~\ref{main.time-reversibility-condition}, reversibility of the mating is governed by the equivalence of $g \circ f$ and $f \circ g$; when the maps are polynomials this connects to classical questions of Ritt \cite{ritt-22, ritt-23} on commuting maps and shared composition factors, and to the coincidence problem for Julia sets \cite{beardon-90, beardon-92, levin-przytycki-97}. Is reversibility of $F$ equivalent to an explicit Ritt-type relation between $f$ and $g$?

\item \emph{Parameter spaces.} The family of pairs $(f,g)$ with $g \circ f \in \mathrm{Per}_1^{pq}(1)$ carries a natural `composition connectedness locus'; does the dynamical homeomorphism of \cite{bullett-lomonaco-2024} between the modular and parabolic Mandelbrot set admit an analogue in this two-map setting?

\item \emph{Pinching.} Can the orbit matings of Theorem~\ref{main.hyperbolic_mating} be pinched, as in \cite{bullett-haissinsky-2007}, to reach representations with parabolic elements, or degenerations of Bers type?
\end{enumerate}

\subsection{Structure of the paper}\label{subsec:structure}

Section~\ref{sec.preliminaries} presents the parabolic classes $\mathcal{P}_d$ and $\mathrm{Per}_1^d(1)$, the groups $\Gamma_{p,q}$, and the definitions concerning correspondences and their iteration. Section~\ref{sec.proof_of_theorem}
proves Theorems~\ref{main.parabolic_mating} and~\ref{main.parabolic_cov}, as well as Proposition \ref{main.time-reversibility-condition}. Section~\ref{sec.hyperbolic_surgery} defines the Kleinian perturbations of $\Gamma_{p,q}$ and proves Theorems~\ref{main.hyperbolic_mating} and~\ref{main.hyperbolic_cov}.



\section{Parabolic maps and groups}\label{sec.preliminaries}

In this section, we introduce the classes $\mathcal{P}_d$ and $\Per_1^d(1)$ of parabolic maps, and the groups $\Gamma_{p,q}$ that are representations of $C_{p+1}\ast C_{q+1}$.

\subsection{Parabolic maps}\label{subsec.parab_maps}

We define $\mathcal{P}_d$ to be the class of degree $d$ rational maps $R$ for which $\infty$ is a fixed point of multiplier $1$ with a unique attracting direction. In $\mathcal{P}_d$, we identify the subclass $\Per_1^d(1)$ of degree $d$ rational maps $R$ that possess a fully invariant immediate basin of attraction $\mathcal{A}(R)$; notice that $\mathcal{A}(R)$ must be connected since $R$ presents a unique attracting direction at $\infty$.

\begin{prop}\label{prop.alignment}

    Let $f \in \mathcal{P}_p, g \in \mathcal{P}_q$ be parabolic maps of degrees $p$ and $q$, respectively, and assume that the degree $pq$ map $g\circ f$ lies in $\Per_1^{pq}(1)$. Then $f\circ g \in \Per_1^{pq}(1)$. Furthermore, there are maps $\tilde{f} \in \mathcal{P}_p, \tilde{g} \in \mathcal{P}_q$ such that $g\circ f$ and $f\circ g$ are affinely conjugate to $\tilde{g}\circ \tilde{f}$ and $\tilde{f}\circ \tilde{g}$, respectively, and the attracting directions at $\infty$ for $\tilde{f}, \tilde{g}, \tilde{g}\circ \tilde{f}$ and $\tilde{f}\circ \tilde{g}$ are all positive real.
    
\end{prop}
\begin{proof}

    Notice that, since $f\in \mathcal{P}_p$ and $g\in \mathcal{P}_q$, we have constants $a_0, b_0 \in \C\setminus\{0\}$ such that
    \[ f(z) = z + a_0 + \mathcal{O}\left( \frac{1}{z} \right) \text{ and } g(z) = z + b_0 + \mathcal{O}\left( \frac{1}{z} \right) \]
    near $\infty$. Thus $g\circ f(z) = z + a_0 + b_0 + \mathcal{O}(1/z)$ and the assumption that $g\circ f \in \Per_1^{pq}(1) \subset \mathcal{P}_{pq}$ guarantees that $a_0 + b_0 \neq 0$. Thus $f\circ g \in \mathcal{P}_{pq}$, and we must show that it presents a fully invariant connected immediate basin of attraction of $\infty$. It will suffice to show that $\mathcal{A}_\infty(f\circ g) = f(\mathcal{A}_\infty(g\circ f))$, where $\mathcal{A}_\infty(R)$ denotes the full basin of attraction of $\infty$ for $R \in \mathcal{P}_d$, since $\mathcal{A}_\infty(g\circ f) = \mathcal{A}(g\circ f)$ is connected by hypothesis.
    
    We begin by noticing that
    \[ z \in \mathcal{A}_\infty(g\circ f) \iff (g\circ f)^n(z) \xrightarrow[n\to \infty]{} \infty \ \text{ along the direction of } a_0 + b_0. \]
    That in turn means that $(f\circ g)^n(f(z)) = f\circ(g\circ f)^n(z) \xrightarrow{} \infty$ along $a_0 + b_0$, since $f$ is tangent to the identity map at $\infty$. This also the attracting direction of $f\circ g(z) = z + a_0 + b_0 + \mathcal{O}(1/z)$. We deduce that $f(z) \in \mathcal{A}_\infty(f\circ g)$, and conclude that $f(\mathcal{A}_\infty(g\circ f)) \subseteq \mathcal{A}_\infty(f\circ g)$. To obtain equality, take any $z \in \mathcal{A}_\infty(f\circ g)$ and any $w \in \widehat{\C}$ such that $z = f(w)$. Then $(g\circ f)^n(w) = g\circ (f\circ g)^{n-1}(z)$, which converges to $\infty$ along the direction of $a_0 + b_0$ since $z \in \mathcal{A}_\infty(f\circ g)$, and $g$ is tangent to the identity map at $\infty$, meaning that $w \in \mathcal{A}_\infty(g\circ f)$. Hence  $f^{-1}(\mathcal{A}_\infty(f\circ g)) \subseteq \mathcal{A}_\infty(g\circ f)$ and so $\mathcal{A}_\infty(f\circ g) \subseteq f(\mathcal{A}_\infty(g\circ f))$.

     Now, set $\varphi(z) := (z + b_0 - a_0)/(b_0 + a_0)$ and $\psi(z) := 2z/(b_0 + a_0)$, and define $\tilde{f} := \varphi\circ f\circ \psi^{-1}$ and $\tilde{g} := \psi\circ g\circ \varphi^{-1}$. We then have $\tilde{f} \in \mathcal{P}_p, \tilde{g} \in \mathcal{P}_q$, and:
     \[ \tilde{f}(z) = z + 1 + \mathcal{O}\left( \frac{1}{z} \right); \quad \tilde{g}(z) = z + 1 + \mathcal{O}\left( \frac{1}{z} \right); \]
     \[ \tilde{g}\circ \tilde{f}(z) = \psi\circ (g\circ f)\circ \psi^{-1}(z) = z + 2 + \mathcal{O}\left( \frac{1}{z} \right); \]
     \[ \tilde{f}\circ \tilde{g}(z) = \varphi\circ(f\circ g)\circ\varphi^{-1}(z) = z + 2 + \mathcal{O}\left( \frac{1}{z} \right) \]
     for $z$ near $\infty$, concluding the proof.
    
\end{proof}

Given a map $R \in \Per_1^d(1)$, one can define its \textit{filled Julia set} to be the complement of the basin of attraction of $\infty$:
\[ K(R) := \widehat{\C}\setminus \mathcal{A}(R). \]

\begin{cor}

    Under the hypotheses of the proposition above, $f(K(g\circ f)) = K(f\circ g)$ and $g(K(f\circ g)) = K(g\circ f)$.
    
\end{cor}

See Remark \ref{rmk.roots} following the proof of Theorem \ref{main.parabolic_mating} for a mild relaxation of these hypothesis.

\begin{rmk}\label{rmk.distinct_julia_sets}

    Although the maps $f$ and $g$ provide branched coverings of the filled Julia sets $K(g\circ f)$ and $K(f\circ g)$ over one another, these sets are in general different. Notice that, since for a map $R \in \Per_1^d(1)$, we have $\partial K(R) = \partial\mathcal{A}(R) = J(R)$, the question of when the filled Julia sets coincide reduces to when the Julia sets coincide. For polynomials $P, Q$ of the same degree, Beardon \cite{beardon-92} has shown that $J(P) = J(Q)$ if and only if there exists an affine map $L$ such that $Q = L\circ P$. The general case for rational maps has been investigated by several authors (for example see \cite{beardon-90}, \cite{levin-przytycki-97}), but we do not know of any complete characterization. It is obvious, however, that $K(g\circ f) = K(f\circ g)$ when the maps $f$ and $g$ commute --- i.e. when $g\circ f = f\circ g$. Questions of composition factors and commutativity were studied by Ritt \cite{ritt-22, ritt-23}.
    
\end{rmk}

In order to construct our matings, we must consider the notion of \textit{pinched polynomial-like maps}. This has been discussed in the literature several times, both informally and formally (for example in \cite{makienko-1993, bullett-freiberger-2005, lyubich-mazor-mukherjee-2024a}), and is related to the notion of parabolic-like maps introduced by the second author \cite{lomonaco-2014, lomonaco-2015}. Here we shall adopt the definition proposed in \cite{bullett-lomonaco-lyubich-mukherjee-2024}. 

A \textit{polygon} is a closed Jordan disk in $\widehat{\C}$ such that its boundary is a piecewise smooth curve; the points of non-smoothness are called the \textit{corners} of the polygon. A \textit{pinched polygon} is a closed connected set in $\widehat{\C}$ composed of finitely many polygons whose pairwise intersections are either empty or comprised of a single point. These intersection points are called the \textit{pinched points} of the pinched polygon, and the corners of its composing polygons are considered corners of the pinched polygon.

\begin{deft}

    A triple $(R, U', U)$ is a \textit{pinched polynomial-like map} if $U, U'$ are open sets such that $\overline{U} \subset \widehat{\C}$ is a polygon, $\overline{U'} \subset \overline{U}$ is a pinched polygon, $\partial U'\cap \partial U$ is the set of corners of $\overline{U}$ and is contained in the set of corners of $\overline{U'}$, and $f: U' \to U$ is a holomorphic map satisfying:
    \begin{itemize}
        \item $f$ is a branched cover from each component of $U'$ onto $U$;
        \item $f$ extends continuously to $\partial U'$;
        \item the corners and pinched points of $\overline{U'}$ are the pre-images under $f$ of the corners of $\overline{U}$.
    \end{itemize}
    Given a pinched polynomial-like map, we can define its \textit{filled Julia set} to be
    \[ K(R) := \bigcap_{n\geq \infty}R^{-n}(U'). \]
    
\end{deft}

We take two copies of the Riemann sphere and define our parabolic maps $f$ and $g$ as maps from the first copy to the second, and the second back to the first, respectively. We can now consider pinched polynomial-like restrictions of $g\circ f$ and $f\circ g$. Let $U$ be a pinched neighborhood of $K(g\circ f)$, pinched at $\infty$ --- i.e. $K(g\circ f) \subset \overline{U}$ and $K(g\circ f)\cap \partial U = \{\infty\}$ --- such that $\overline{U}$ is a polygon with a single corner at $\infty$. If $U$ is a small enough neighborhood, then $(g\circ f)^{-1}(U) =: U'$ is an open set such that $\overline{U'}$ is a pinched polygon containing $K(g\circ f)$, and $(g\circ f, U', U)$ is a pinched polynomial-like map. It will be important for our construction that at the pinch point $\infty$ of $U$ the external angle $\theta$ made by $\partial U$ is strictly greater than $0$ (Figure \ref{fig:pinched_polynomial_like_restrictions}). We shall determine this angle later, but it can be taken as small as desired.

Consider the set $f(U') = g^{-1}(U)$. Its closure $\overline{f(U')}$ is another pinched polygon, now containing $K(f\circ g)$. We next choose $V$ a pinched neighborhood of $K(f\circ g)$, pinched at $\infty$, such that $\overline{V}$ is a polygon with a single corner at $\infty$, and, if we set $V' := (f\circ g)^{-1}(V)$, we have $\overline{V'} \subset \overline{f(U')} \subset \overline{V}$. To see that we can do this, simply consider the pre-Fatou coordinates at $\infty$ for the sphere on which $f\circ g$ acts. The boundary of $f(U')$ is a curve that makes an angle of $\theta$ at $\infty$, just like $U$ does, since $f$ is tangent to identity there. We can then take $V$ so that its boundary, in a neighborhood of $\infty$, makes the same angle to the left of $\partial f(U')$, while its image under $f\circ g(z) = z + 2 + \mathcal{O}(1/z)$ falls entirely to the right of $\partial f(U')$. Avoiding intersection away from infinity can be attained by simple perturbations. Notice that $(f\circ g, V', V)$ is a pinched polynomial-like map, and that $\overline{U'} \subset \overline{g(V')} = \overline{f^{-1}(V)} \subset \overline{U}$. See Figure \ref{fig:pinched_polynomial_like_restrictions}. Notice also that $\overline{U'}$ and $\overline{V'}$ both have $pq$ pinch points and corners, as that is the number of pre-images of $\infty$ under $g\circ f$ and $f\circ g$, while $\overline{f(U')}$ and $\overline{g(V')}$ are pinched polygons with $q$ and $p$ pinch points and corners, respectively.

\begin{figure}
    \centering
    \includegraphics[width=0.9\linewidth]{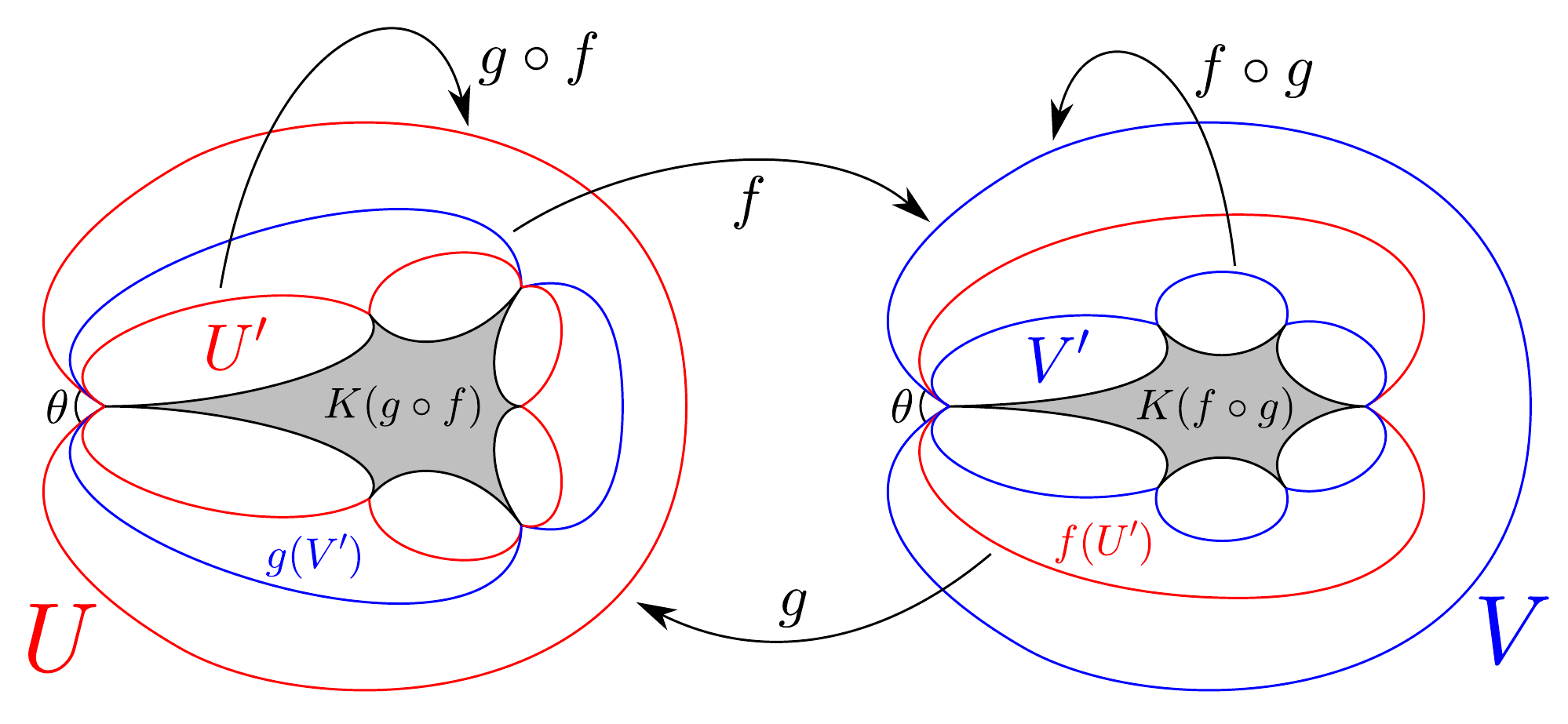}
    \caption{The pinched polynomial-like restrictions of $g\circ f$ and $f\circ g$, for $p = 3$ and $q = 2$. The actions of $f$ and $g$ induce intermediary pinched polygons in the picture.}
    \label{fig:pinched_polynomial_like_restrictions}
\end{figure}

\subsection{The groups}

We shall work in the Poincar\'e disk model $\D$ of the hyperbolic plane. We fix a pair of integers $p,q \geq 2$, and consider groups of orientation-preserving isometries of $\D$, generated by pairs $(\sigma,\rho)$ conjugate to counter-clockwise rotations through angles $2\pi/(p+1)$ and $2\pi/(q+1)$ respectively (in the upper half-plane model of the hyperbolic plane these are subgroups of $PSL(2,\R)$). Up to M\"obius conjugacy such groups are parametrized by the hyperbolic distance between the fixed points of $\sigma$ and $\rho$: there is just one value of this distance for which $\sigma\circ\rho$ is parabolic, and then $\rho\circ\sigma$ is automatically parabolic too. Without loss of generality we may take the geodesic on which the fixed points of $\sigma$ and $\rho$ lie to be the horizontal diameter of $\D$. We now define an explicit group $\Gamma_{p,q}$ as follows. Let $P \in (-1, 0)$ be the unique point on the diameter of $\D$ such that the geodesic rays starting at $P$ and landing at $i$ and $-i$ form an angle of $2\pi/(p + 1)$ with each other, and let $Q \in (0, 1)$ be the point such that the geodesic rays starting at $Q$ and landing at $i$ and $-i$ form an angle of $2\pi/(q + 1)$ with each other. Then we take $\rho$ to be the rotation of $2\pi/(p + 1)$ about $P$, and $\sigma$ to be the rotation of $2\pi/(q + 1)$ about $Q$. The domain $\Delta_\rho \subset \D$ of points to the right of the geodesic rays starting at $P$ and ending at $\pm i$ is a fundamental domain for the action of $\rho$ on $\D$, and similarly the domain $\Delta_\sigma \subset \D$ of points to the left of the geodesic rays starting at $Q$ and ending at $\pm i$ is a fundamental domain for the action of $\sigma$ on $\D$. By Klein's Combination Theorem (see for instance \cite{maskit-1965}), the intersection $\Delta := \Delta_\rho\cap \Delta_\sigma$ is a fundamental domain for the action of $\Gamma$. See Figure \ref{fig:parabolic_fundamental_domains}. The points $i$ and $-i$ are then the 
fixed points of the compositions $\sigma\rho$ and $\rho\sigma$, and both these compositions are parabolic since their fixed points are on the boundary of $\D$.

\begin{figure}
    \centering
    \includegraphics[width=0.25\linewidth]{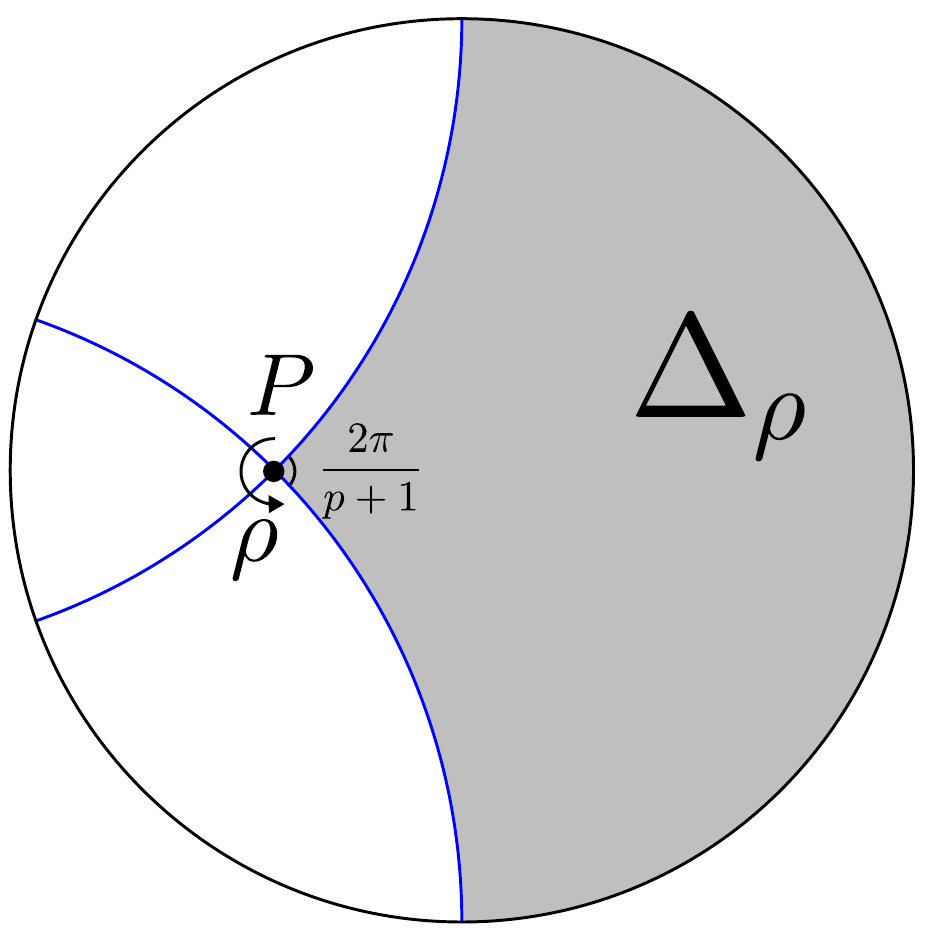}\hspace{0.05\linewidth}
    \includegraphics[width=0.25\linewidth]{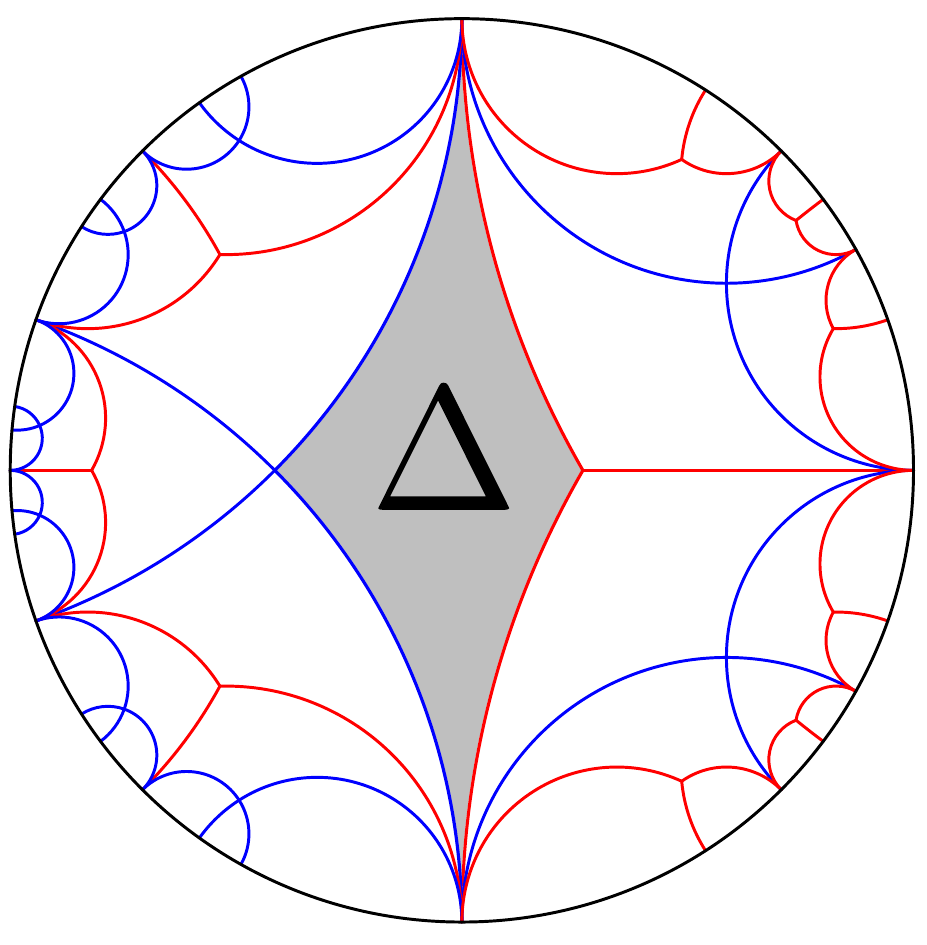}\hspace{0.05\linewidth}
    \includegraphics[width=0.25\linewidth]{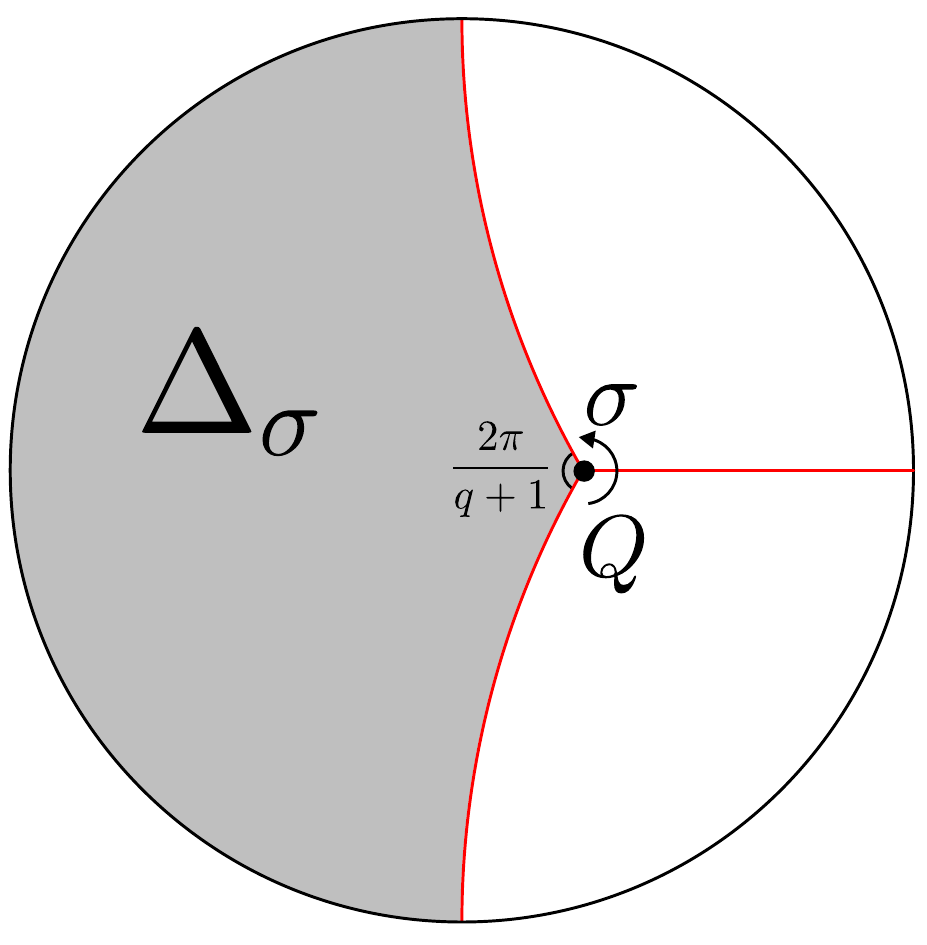}
    \caption{On the left, the fundamental domain for $\rho$, and on the right, the fundamental domain for $\sigma$. In the middle, the intersection between these domains, a fundamental domain for $\Gamma_{p,q} = \left< \rho, \sigma \right>$, together with some of its images under elements of the group.}
    \label{fig:parabolic_fundamental_domains}
\end{figure}

\begin{rmk}

    The group $\Gamma_{p,q}$ is the orientation preserving subgroup of the triangle group $\Delta(p+1, q+1, \infty)$ --- with the hyperbolic triangle of vertices $P, Q$ and $i$ a fundamental domain for the triangle group. Note that while reflection in the real axis conjugates $\sigma\rho$ to $\sigma^{-1}\rho^{-1}$, this symmetry is not an element of $\Gamma_{p,q}$.
    
\end{rmk}

Let us now take $I$ the geodesic from $i$ to $-i$; notice that $I \subset \Delta$. In particular, $I \subset \Delta_\sigma$ and therefore its images $\sigma^nI$, $1\leq n\leq q$, lie in $\D\setminus \Delta_\sigma \subset \Delta_\rho$. Also, since $I$ connects $i$ to $-i$, which are the extremal points of $\s^1$ that belong to $\overline{\Delta_\sigma}$, the curves $\sigma^nI$ connect the extremal points of each of the images of $\Delta_\sigma$ and, therefore, cover the whole half-circle to the right of $\Delta_\sigma$. Combining these two facts, we see that the curves $\rho^m\sigma^nI$, $1\leq n\leq q$ and $1\leq m\leq p$, all lie to the left of $I$ and cover the left half of the circle. Thus the union of the maps $(\rho^m\sigma^n)^{-1} = \sigma^{q+1-n}\rho^{p+1-m}$, each applied to the corresponding segment $\rho^m\sigma^nI$, forms a $pq:1$ map of the left half of the circle onto $I$, similar to the the pinched polynomial-like map from $\partial U'$ to $\partial U$. In fact, it is easy to see that, if one identifies the points $i$ and $-i$ on the right half of the disk, one ends up with a true pinched polynomial-like map. The same can be said for the union of the $(\sigma^n\rho^m)^{-1}$ on the curves $\sigma^n\rho^mI$, $1\leq n\leq q$ and $1\leq m\leq p$, that cover the right half of the circle. See Figure \ref{fig:group_pinched_polynomial_like}. Notice that the curves $\rho^m I$, $1 \leq m\leq p$, and $\sigma^n I$, $1\leq n\leq q$, form intermediary curves in these pinched polynomial-like maps, similar to the picture for the maps $g\circ f$ and $f\circ g$.

\begin{figure}
    \centering
    \includegraphics[width=0.4\linewidth]{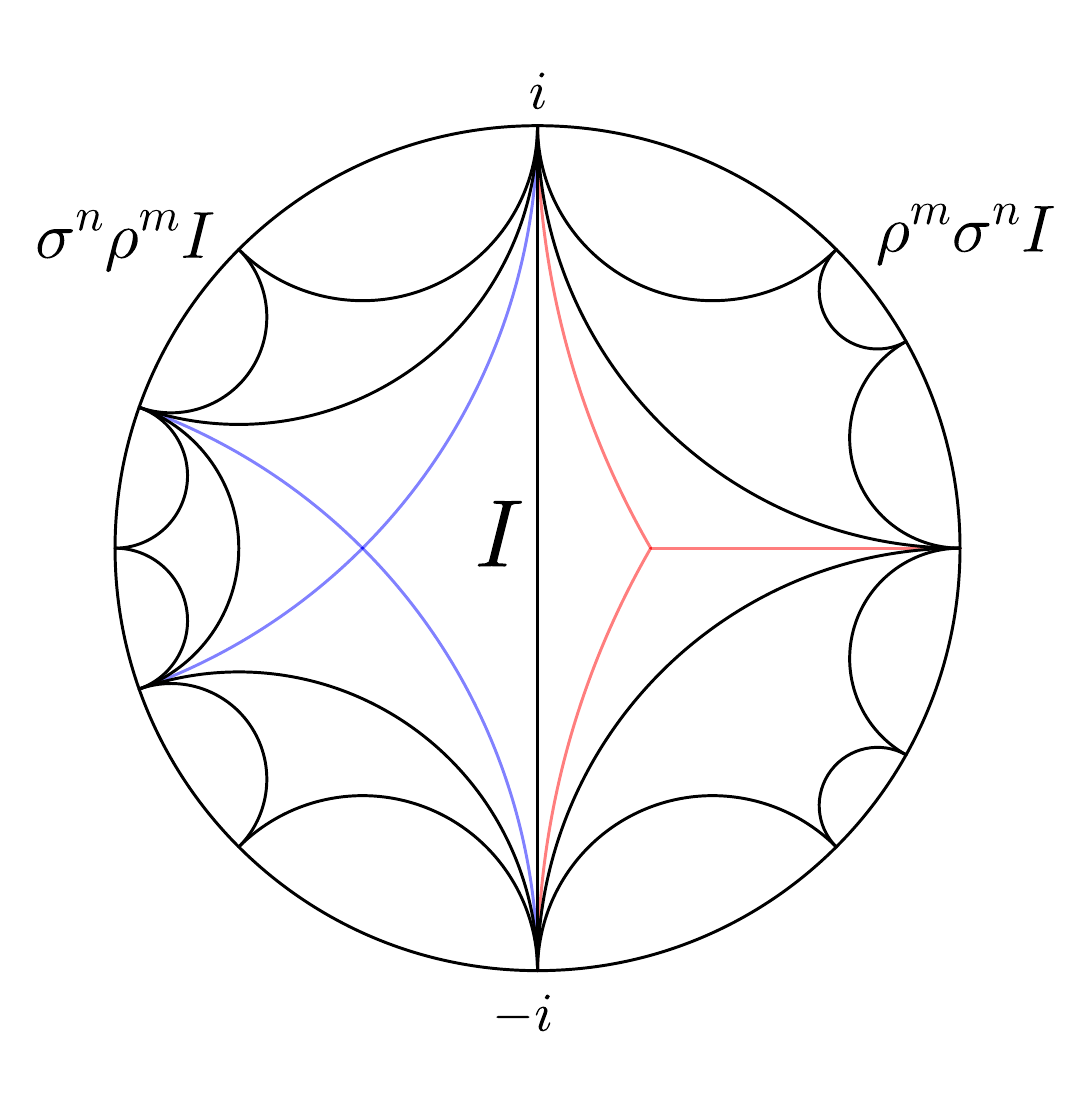}
    \caption{The geodesic $I$ and its images under different elements of $\Gamma_{p,q}$. The curves in blue and red are left from the tessellation of Figure \ref{fig:parabolic_fundamental_domains}}.
    \label{fig:group_pinched_polynomial_like}
\end{figure}

When proving Theorem \ref{main.parabolic_mating}, we will need to use different curves to remove certain cusps from the picture, but the main idea is still that this picture of the group is compatible with the pinched polynomial-like restrictions of the maps.

\subsection{Correspondences}\label{subsec.correspondences}

We now introduce the concepts that are the generalizations of `maps' and `groups' which we use to realize a mating. A more detailed treatment can be found in the article \cite{bullett-penrose-2001} by the first author and C. Penrose.

\begin{deft}

    An \textit{algebraic correspondence} $F$ on $\widehat{\C}$ is an algebraic curve
    \[ F = \{ (z, w) \in \widehat{\C}\times \widehat{\C} \ | \ P(z, w) = 0 \} \]
    where $P$ is a polynomial which has no factors involving just one of the variables, and no repeated factors involving both variables. The \textit{image} of a point $z \in \widehat{\C}$ is defined as the set of points:
    \[ F(z) := \{ w \in \widehat{\C} \ | \ (z, w) \in F \}. \]
    
\end{deft}

Correspondences are to be understood as multivalued maps. As such, the following are some of their properties:

\begin{itemize}
    \item If $F$ is a correspondence, there is an associated inverse correspondence
    \[ F^{-1} := \{ (z, w) \in \widehat{\C}\times \widehat{\C} \ | \ P(w, z) = 0 \}. \]
    We do not have $F\circ F^{-1} = Id$, but viewing the identity as a correspondence (it is just the diagonal), then $Id \subseteq F\circ F^{-1}$;

    \item A correspondence has a bidegree: for generic $z^\ast$ and $w^\ast$ in $\widehat{\C}$, the sets $F^{-1}(w^\ast)$ and $F(z^\ast)$ always have some integer number of elements, say $d_1$ and $d_2$, respectively, and this pair of numbers is the \textit{bidegree} of $F$. We then say that $F$ is a $d_1:d_2$ correspondence. Notice that, since $F$ has no repeated factors in both variables, $d_1$ and $d_2$ are the degrees of the polynomials $z\mapsto P(z, w^\ast)$ and $w\mapsto P(z^\ast, w)$ for generic choices of $z^\ast$ and $w^\ast$;

    \item Two correspondences can be composed: if $F, G$ are correspondences then
    \[ G\circ F := \{ (z, w) \in \widehat{\C}\times \widehat{\C} \ | \ \exists v \in \widehat{\C}, (z, v) \in F, (v, w) \in G \}. \]
    It is not immediately obvious that such a set is an algebraic curve, but one can show it is a complex submanifold and apply Chow's Theorem \cite{chow-1949}. In general, if $F$ has bidegree $(d_1, d_2)$ and $G$ has bidegree $(d_1', d_2')$, then $G\circ F$ has bidegree \textit{at most} $(d_1d_1', d_2d_2')$ (but not necessarily the same).

    \item The \textit{union} of two correspondences is also a correspondence. If $F$ and $G$ are correspondences defined by the polynomials $P(z, w)$ and $Q(z, w)$, respectively, and $P = P'R$ and $Q = Q'R$, with $P'$ and $Q'$ having no factor in common, then it is clear that $F\cup G$ is defined by $P'Q'R$ --- which will have no repeated factors. If $(d_1,d_2)$ is the bidegree of $F$, $(d_1',d_2')$ is the bidegree of $G$, and $(\hat{d_1},\hat{d_2})$ is the bidegree of $R$, then the bidegree of $F\cup G$ is $(d_1 + d_1' - \hat{d_1}, d_2 + d_2' - \hat{d_2})$. Furthermore, if $F\subset G$ and $F \neq G$, we can talk about the \textit{difference} of correspondences $G\setminus F$, of bidegree $(d_1' - d_1, d_2' - d_2)$.
\end{itemize}

See Section 2 of \cite{bullett-penrose-2001} for equivalent alternative definitions of a \textit{holomorphic correspondence} on the Riemann sphere, and details of the application of Chow's Theorem in this context. Of particular relevance to us are the deleted covering correspondences.

\begin{deft}

    Given a rational map $R$ on $\widehat{\C}$, we define its \textit{deleted covering correspondence} as
    \[ \Cov_0^R := \left\{ (z, w) \in \widehat{\C}\times \widehat{\C} \ \left| \ \frac{R(w) - R(z)}{w - z} = 0 \right.\right\}. \]
    
\end{deft}

The deleted covering correspondence simply moves a point $z$ to all \textit{other} points in the fiber of $R(z)$ (notice that $z \in \Cov_0^R(z)$ if and only if the local degree of $R$ at $z$ is greater than $1$). With the notation introduced above, we can see that $\Cov_0^R = R\circ R^{-1}\setminus Id$, and, if $d$ is the degree of $R$, $\Cov_0^R$ has bidegree $(d-1, d-1)$.

\begin{deft}\label{deft.parabolic_mating}

    Let $f \in \mathcal{P}_p$, $g \in \mathcal{P}_q$ be such that $g\circ f \in \Per_1^{pq}(1)$. We say a $pq:pq$ correspondence $F$ is a \textit{mating} between the pair $(f, g)$ and the group $\Gamma_{p,q}$ if there exists a decomposition $\widehat{\C} = \Lambda\cup \Omega$ into fully $F$-invariant sets such that:
    \begin{itemize}
        \item $\Lambda = \Lambda_-\cup \Lambda_+$ is the union of two compact sets, with $\Lambda_-\cap \Lambda_+$ consisting of a single point, such that the $pq:1$ restriction $F|_{\Lambda_-}: \Lambda_- \to \Lambda_-$ admits a pinched polynomial-like restriction hybrid conjugate to $g\circ f$, and the $pq:1$ restriction $F^{-1}|_{\Lambda_+}: \Lambda_+ \to \Lambda_+$ admits a pinched polynomial-like restriction hybrid conjugate to $f\circ g$;

        \item $\Omega$ is an open set, and the action of $F|_{\Omega}$ is conformally conjugate to the action of the elements $\sigma^n\rho^m$, $1\leq n\leq q$, $1\leq m\leq p$, of $\Gamma_{p,q}$ on the unit disk $\D$.
    \end{itemize}
    
\end{deft}

\begin{rmk}\label{rmk.definition_by_quotient}

    The second item in the previous definition implies that the quotient $\faktor{\Omega}{F}$ is biholomorphic to $\faktor{\D}{\Gamma_{p,q}}$. This serves as inspiration for the weaker definition of orbit mating (see Definition \ref{deft.hyperbolic_mating}).
    
\end{rmk}

\begin{rmk}\label{rmk.hecke}

    When defining the group $\Gamma_{p,q}$, one could allow $q = 1$ and $p \geq 2$ by making $\sigma(z) = -z$. These groups $\Gamma_{p,1}$ are conjugate to the Hecke groups $H_{p+1}$. In the given definition of mating, one would require $g$ to be a degree $1$ map, i.e. a M\"{o}bius transformation. Then clearly $g\circ f$ and $f\circ g$ are conjugate via $g$, meaning that the filled Julia sets are essentially the same. With that in mind, it is clear that a mating between the pair $(f, g)$ and the group $H_{p+1}$ in this sense recovers the definition of mating between $g\circ f$ and $H_{p+1}$ given in \cite{bullett-lomonaco-2019, bullett-lomonaco-2022, bullett-lomonaco-2024, bullett-lomonaco-lyubich-mukherjee-2024}. Furthermore, a mating between a map $f$ and $H_{d+1}$ in the sense of those articles is a mating between the pair $(f, Id)$ and $H_{d+1}$ in the present setting.
    
\end{rmk}


\section{Proof of Theorems \ref{main.parabolic_mating} and \ref{main.parabolic_cov}}\label{sec.proof_of_theorem}

Throughout this section, we always assume $f \in \mathcal{P}_p$, $g \in \mathcal{P}_q$, $g\circ f \in \Per_1^{pq}(1)$. The result is proved using the following steps:

\begin{itemize}
    \item We adapt the `pinched polynomial-like maps' we defined for $\Gamma_{p,q}$, to widen certain cusps in their domain boundaries into non-zero angles;
    \item We impose the same angle conditions on the boundaries of the domains of the pinched polynomial-like restrictions of the maps $g\circ f$ and $f\circ g$;
    \item We find the conditions necessary for a diffeomorphism of these domain boundaries to correctly glue together the actions of the group and the maps;
    \item Once we have defined the diffeomorphisms on boundaries, we interpolate quasiconformally to interiors, using estimates of Pommerenke and Warschawski;
    \item We construct a topological version of the correspondence;
    \item We find an invariant Beltrami form, and by applying the Measurable Riemann Mapping Theorem we straighten the topological mating into a conformal mating.
\end{itemize}

The first step is to open up certain cusps on boundaries in the group picture, to modify them to non-zero angles. Consider 
two smooth curves, $\gamma\subset D_\sigma$ and $\eta\subset D_\rho$, each joining $i$ to $-i$, and each making an angle $\theta\ne 0$ with the dividing geodesic $I$ at both extremities (see Figure \ref{fig:group_pinched_polynomial_like_angle}); we want this angle $\theta$ to be the same as the one the pinched polynomial-like restrictions of $g\circ f$ and $f\circ g$ make at $\infty$. Now define: $\gamma_n := \sigma^{-n}\gamma$ for $1\leq n\leq q$; $\eta_m := \rho^{-m}\eta$ for $1\leq m\leq p$; $\gamma_{n,m} := \rho^{-m}\sigma^{-n}\gamma$ for $1\leq n\leq q, 1\leq m\leq p$; and $\eta_{m,n} := \sigma^{-n}\rho^{-m}\eta$ for $1\leq m\leq p, 1\leq n\leq q$. Let us denote by $A$ the set between $\gamma$ and the union of the arcs $\gamma_{n,m}$. If we were to identify $i$ and $-i$, $A$ would become a pinched annulus, with inner boundary $\bigcup_{n,m} \gamma_{n,m}$ and outer boundary $\gamma$; we shall write $\partial_i A := \bigcup_{n,m} \gamma_{n,m}$ and $\partial_o A := \gamma$. Similarly, we define $B$ to be the set bounded by $\eta$ and the union of the arcs $\eta_{m,n}$, and we let $\partial_i B := \bigcup_{m,n} \eta_{m,n}$ and $\partial_o B := \eta$. Since $\sigma^n\rho^m$ maps $\gamma_{n,m}$ to $\gamma$, these group elements together make a $pq:1$ map from $\partial_i A$ to $\partial_o A$, just as $g\circ f$ is a $pq:1$ map from $\partial U'$ to $\partial U$. Similarly, since $\sigma^{-n}\rho^{-m}$ maps $\eta$ to $\eta_{m,n}$, these group elements together define a $1:pq$ correspondence from $\partial_o B$ to $\partial_i B$, just as $(f\circ g)^{-1}$ induces a $1:pq$ correspondence from $\partial V$ to $\partial V'$. Our objective now is to construct quasiconformal maps $\varphi: A \to U\setminus \overline{U'}$ and $\psi: B \to V\setminus \overline{V'}$ which are conjugacies between the maps from inner to outer boundaries defining our two dynamical systems:
\[ \begin{tikzcd}
\partial_iA \arrow[d, "\varphi"'] \arrow[rr, "{\{\sigma^n\rho^m \}_{n,m}}"] &  & \partial_oA \arrow[d, "\varphi"] &  & \partial_iB \arrow[d, "\psi"'] \arrow[rr, "{\{ \rho^m\sigma^n \}_{m,n}}"] &  & \partial_oB \arrow[d, "\psi"] \\
\partial U' \arrow[rr, "g\circ f"']                                                       &  & \partial U                       &  & \partial V' \arrow[rr, "f\circ g"']                                                     &  & \partial V                   
\end{tikzcd} \]
We now also set $\hat{U} := g(V') = f^{-1}(V) \subset U$, $\hat{V} := f(U') = g^{-1}(U) \subset V$, $\hat{A}$ the set bounded by $\gamma$ and the union of the $\gamma_n$, and $\hat{B}$ the set bounded by $\eta$ and the union of the $\eta_m$. Just as before, if we set $\partial_i\hat{A} := \bigcup_n \gamma_n$ and $\partial_i\hat{B} := \bigcup_m \eta_m$, we find: a $p:1$ map from $\partial_iA$ to $\partial_i\hat{B}$ defined by the iterates of $\rho$, like the $p:1$ map $f:\partial U' \to \partial\hat{V}$; a $q:1$ map from $\partial_i\hat{B}$ to $\gamma$ defined by the iterates of $\sigma$, like the $q:1$ map $g:\partial\hat{V} \to \partial U$; a $q:1$ map from $\partial_i B$ to $\partial_i\hat{A}$ defined by the iterates of $\sigma$, like the $q:1$ map $g:\partial V' \to \partial\hat{U}$; and a $p:1$ map from $\partial_i\hat{A}$ to $\eta$ defined by the iterates of $\rho$, like the $p:1$ map $f:\partial\hat{U} \to \partial V$. See Figure \ref{fig:group_pinched_polynomial_like_angle}. It will be important for our construction that the maps $\varphi, \psi$ also conjugate dynamical maps between inner and outer boundaries, meaning that we want the following diagrams to commute:
\begin{equation}\label{eq.commutative_diagrams} \begin{tikzcd}
\partial_i A \arrow[r, "\{\rho^m\}_m"] \arrow[d, "\varphi"'] & \partial_i\hat{B} \arrow[d, "\psi"] \arrow[r, "\{\sigma^n\}_n"] & \gamma \arrow[d, "\varphi"] & \partial_i B \arrow[d, "\psi"'] \arrow[r, "\{\sigma^n\}_n"] & \partial_i \hat{A} \arrow[r, "\{\rho^m\}_m"] \arrow[d, "\varphi"] & \eta \arrow[d, "\psi"] \\
\partial U' \arrow[r, "f"']                                  & \partial\hat{V} \arrow[r, "g"']                                       & \partial U                  & \partial V' \arrow[r, "g"']                                       & \partial \hat{U} \arrow[r, "f"']                                  & \partial V            
\end{tikzcd} \end{equation}
Note that this will imply that the two previous diagrams commute.

\begin{figure}
    \centering
    \includegraphics[width=0.35\linewidth]{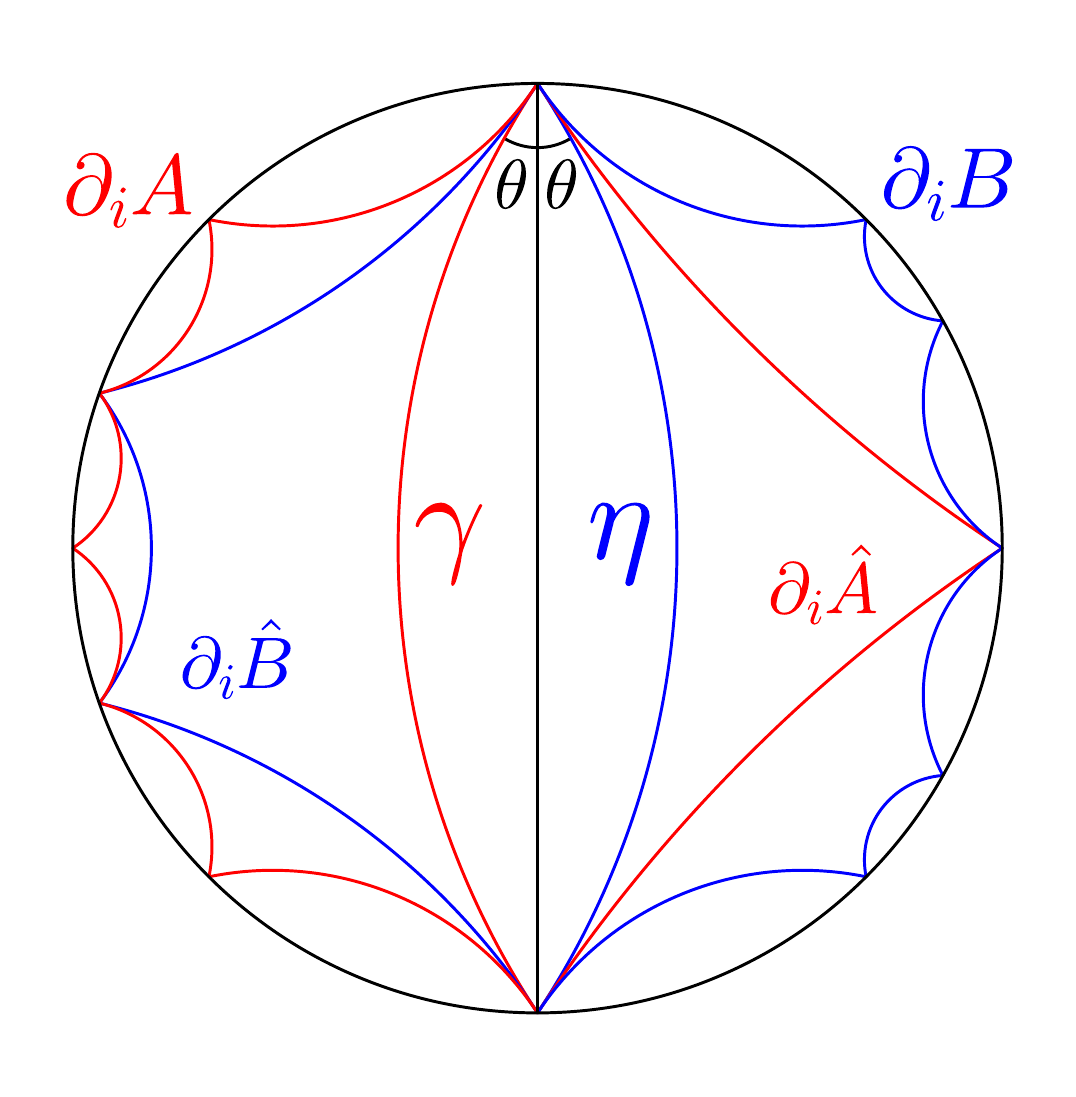}
    \caption{The adapted picture of the group $\Gamma_{p,q}$.}
    \label{fig:group_pinched_polynomial_like_angle}
\end{figure}

\begin{lema}\label{lema.diagram_lift}

    If $\varphi: \gamma \to \partial U$ and $\psi: \eta \to \partial V$ are any chosen diffeomorphisms mapping $\pm i$ to $\infty$, then one can extend $\varphi: \partial_iA\cup \partial_i\hat{A} \to \partial U'\cup \partial\hat{U}$ and $\psi: \partial_iB\cup \partial_i\hat{B} \to \partial V'\cup \partial\hat{V}$ so that the diagrams \ref{eq.commutative_diagrams} commute, and these extensions are diffeomorphisms on each smooth segment of these boundaries.
    
\end{lema}
\begin{proof}

    This follows at once from the degree compatibility of the horizontal maps in each square: we use the diagrams we wish to commute to lift the initial diffeomorphisms $\varphi$ and $\psi$ to each level. The lifts will be diffeomorphisms on the smooth segments since the maps $f, g, \rho^m, \sigma^n$ are all smooth on those segments.
    
\end{proof}

We now wish to extend the diffeomorphisms that we have just constructed on boundaries, to quasiconformal maps $\varphi: A\to U\setminus \overline{U'}$ and $\psi: B\to V\setminus \overline{V'}$. The argument applies the same methods as the first proof of Lemma 5.2 in \cite{bullett-lomonaco-lyubich-mukherjee-2024}, and, earlier, in \cite{lee-lyubich-makarov-mukherjee-2021}. The main difference here is that we do not work with external classes, as we want to preserve the nature of $g\circ f$ and $f\circ g$ as compositions of maps.

\begin{lema}\label{lema.warschawski}

    There exist choices of homeomorphisms $\varphi: \gamma \to \partial U$ and $\psi: \eta \to \partial V$ that admit quasiconformal extensions $\varphi: A\to U\setminus \overline{U'}$ and $\psi: B\to V\setminus \overline{V'}$, and for which the diagrams (\ref{eq.commutative_diagrams}) commute.
    
\end{lema}
\begin{proof}

    Let us define the half-disks $\D_\ell := \{ x + iy \in \D \ | \ x \leq 0 \}$ and $\D_r := \{ x + iy \in \D \ | \ x \geq 0 \}$; notice that $A \subset \D_\ell$ and $B \subset \D_r$. The map $z \mapsto -z^2$ maps both half-disks to the whole of $\D$, sending $\pm i$ to $1$. Let $X$ be the monogon bounded by $\partial\D$ and the image of $\partial_oA = \gamma$, and $X'$ be the pinched polygon bounded by $\partial\D$ and the image of $\partial_iA$. The group elements $\sigma^n\rho^m$, $1\leq m\leq p, 1\leq n\leq q$, then induce a pinched polynomial-like map $\mathcal{E}_1: X' \to X$ (indeed, this is a Farey-like map as defined in \cite{bullett-lomonaco-lyubich-mukherjee-2024}); see Figure \ref{fig:farey_like}. Similarly, if $Y$ is the monogon bounded by $\partial\D$ and the image of $\partial_oB = \eta$, and $Y'$ is the pinched polygon bounded by $\partial\D$ and the image of $\partial_iB$, then the group elements $\rho^m\sigma^n$, $1\leq m\leq p, 1\leq n\leq q$, induce a pinched polynomial-like map $\mathcal{E}_2: Y' \to Y$. Notice that the map $z \mapsto -z^2$ is a conformal map between $A$ and $X\setminus \overline{X'}$, and also between $B$ and $Y\setminus \overline{Y'}$, and therefore it is enough to construct the quasiconformal maps $\hat{\varphi}: X\setminus \overline{X'} \to U\setminus \overline{U'}$ and $\hat{\psi}: Y\setminus \overline{Y'} \to V\setminus \overline{V'}$. To simplify notation, we shall still denote the closure of $\partial X\setminus \s^1$ by $\partial X$, and similarly for $X', Y, Y'$.

    \begin{figure}
        \centering
        \includegraphics[width=0.6\linewidth]{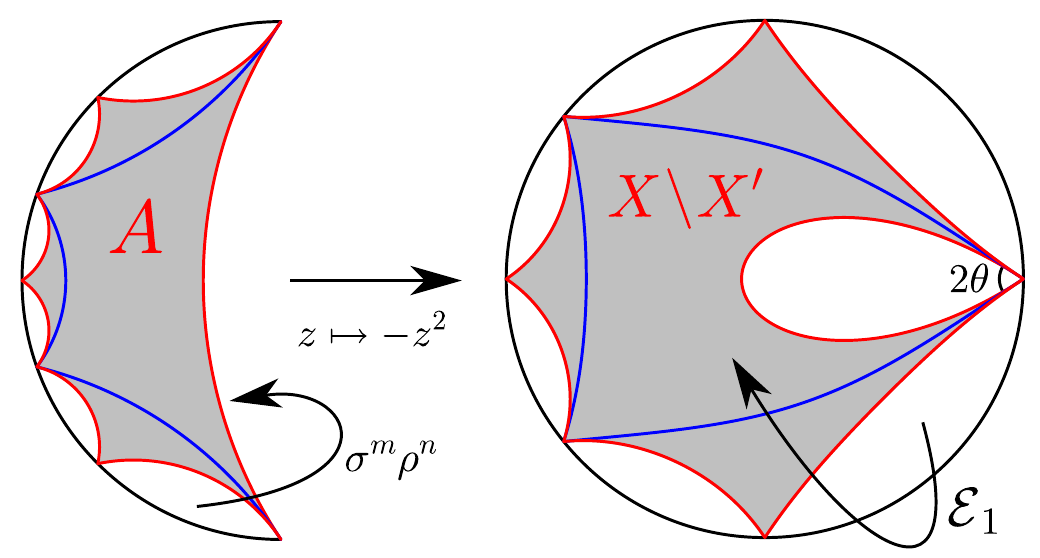}
        \caption{The Farey-like map $\mathcal{E}_1: X' \to X$ induced from the action of the elements $\sigma^m\rho^n$ on the left side of the unit disk. The gray regions show the sets $A$ and $X\setminus X'$. Notice that the angle outside of $X$ is $2\theta$.}
        \label{fig:farey_like}
    \end{figure}
    
    We begin by observing that $\partial X$ and $\partial Y$ make an angle of $2\theta$ at $1$. Recall that $\partial U$ and $\partial V$ make angles of $\theta$ at $\infty$ --- in fact, we could take these sets to be composed of straight lines in a neighborhood of infinity. Then, if we set $\varphi: \partial X \to \partial U$ to be the restriction of a conformal map between $\D\setminus X$ and $\widehat{\C}\setminus U$ mapping $1$ to $\infty$, and $\psi: \partial Y \to \partial V$ to be the restriction of a conformal map between $\D\setminus Y$ and $\widehat{\C}\setminus V$ mapping $1$ to $\infty$, Theorem 3.11 of \cite{pommerenke-1992} tells us that $\varphi$ and $\psi$ behave like
    \[ z \mapsto 1 + b(1 - z)^{1/2} + o((1 - z)^{1/2}) \]
    near $1$, for some $b \in \C\setminus \{0\}$ (that will not usually be the same for $\varphi$ and $\psi$) and some appropriate branch of the square root.

    We may also define $\hat{X} \subset X$ and $\hat{Y} \subset Y$ via the intermediary boundaries $\partial_i\hat{A}$ and $\partial_i\hat{B}$, and there are induced maps $\hat{X} \to Y$ and $\hat{Y} \to X$ obtained from the actions of $\rho^m$, $1\leq m\leq p$, and $\sigma^n$, $1\leq n\leq q$, respectively. Still using the same boundary notation as we do for $X,X',Y,Y'$, Lemma \ref{lema.diagram_lift} tells us that our chosen $\varphi$ and $\psi$ extend to maps $\varphi: \partial X'\cup \partial \hat{X} \to \partial U'\cup \partial \hat{U}$ and $\psi: \partial Y'\cup \partial \hat{Y} \to \partial V'\cup \partial \hat{V}$ such that the following diagrams commute:
    
    \[ \begin{tikzcd}
    \partial X' \arrow[r, "\{\rho^m\}_m"] \arrow[d, "\varphi"'] & \partial \hat{Y} \arrow[d, "\psi"] \arrow[r, "\{\sigma^n\}_n"] & \partial X \arrow[d, "\varphi"] & \partial Y \arrow[d, "\psi"'] \arrow[r, "\{\sigma^n\}_n"] & \partial \hat{X} \arrow[r, "\{\rho^m\}_m"] \arrow[d, "\varphi"] & \partial Y \arrow[d, "\psi"] \\
    \partial U' \arrow[r, "f"']                                  & \partial\hat{V} \arrow[r, "g"']                                       & \partial U                  & \partial V' \arrow[r, "g"']                                       & \partial \hat{U} \arrow[r, "f"']                                  & \partial V            
    \end{tikzcd} \]
    
    We will show that $\varphi$ admits a quasiconformal extension to $X\setminus \overline{X'}$, and the proof for $\psi$ is analogous. The curve $\partial \hat{X}$ cuts the pinched annulus $X\setminus \overline{X'}$ into $p + 1$ tiles. Notice that the set $X\setminus \overline{\hat{X}}$ is one such tile, and all other $p$ tiles share their boundaries with it. Since the boundaries of all of these tiles are piece-wise smooth, classic quasiconformal extension arguments such as the Ahlfors-Beurling extension suffice to extend the map $\varphi$ away from the cusps of the tiles. Thus, we only need to show that a quasiconformal extension is possible at each cusp. We will focus on one of the cusps of $X\setminus \overline{\hat{X}}$ for simplicity, as the proofs for all others follow in the same way, in each case using the fact that the cusps arise from parabolic fixed points of $f$ and $g$, so we can use pre-Fatou coordinates.

    We begin by noting that $\partial V$ can be taken, near $\infty$, to be $\partial U + C_1$ for some $C_1 > 0$ not equal to $1$. With that
    assumption, one finds that $\partial U$ is the image of $\partial \hat{V}$ under the composition of the maps $z \mapsto g(z) = z + 1 + \mathcal{O}(1/z)$ and $z \mapsto z - C_1$, and therefore the cusps of $U\setminus \overline{\hat{U}}$ resemble half-strips of width $1 - C_1$. Let us work with just one of these strips, call it $S_1$; we can uniformize it to the half-strip $S := \{ x + iy \in \C \ | \ x > 0, -\pi/2 < y < \pi/2 \}$ via a map $k_1: S_1 \to S$. A result of Warschawski \cite{warschawski-1942} then tells us that $k_1(z) = c_1z + o(z)$ as $z$ goes to $\infty$. Similarly, when considering the associated cusp of $X\setminus \overline{\hat{X}}$, we may conformally map it to a strip $S_2$ via a map $j$ by first returning to the original disk (applying $z \mapsto z^{1/2}$) and then using coordinates where the group element that fixes the cusp (either $\sigma\rho$ or $\rho\sigma$) is the map $z \mapsto z + 1$. Thus, Warschawski again tells us that the uniformization map $k_2: S_2 \to S$ has the form $k_2(z) = c_2z + o(z)$ as $z$ goes to $\infty$. See Figure \ref{fig:warschawski}. We then see that the map
    \[ k_1\circ \varphi\circ j\circ k_2^{-1}: \partial S \to \partial S \]
    is of the form $z \mapsto \lambda z + o(z)$ as the real part of $z$ goes to $\infty$, and the maps on the upper and lower boundaries of $S$ are at a bounded distance from each other. This means that the linear extension of this map to $S$ will be quasiconformal, and thus induce a quasiconformal extension of $\varphi$ on the cusp.
    
\end{proof}

\begin{figure}
    \centering
    \includegraphics[width=0.8\linewidth]{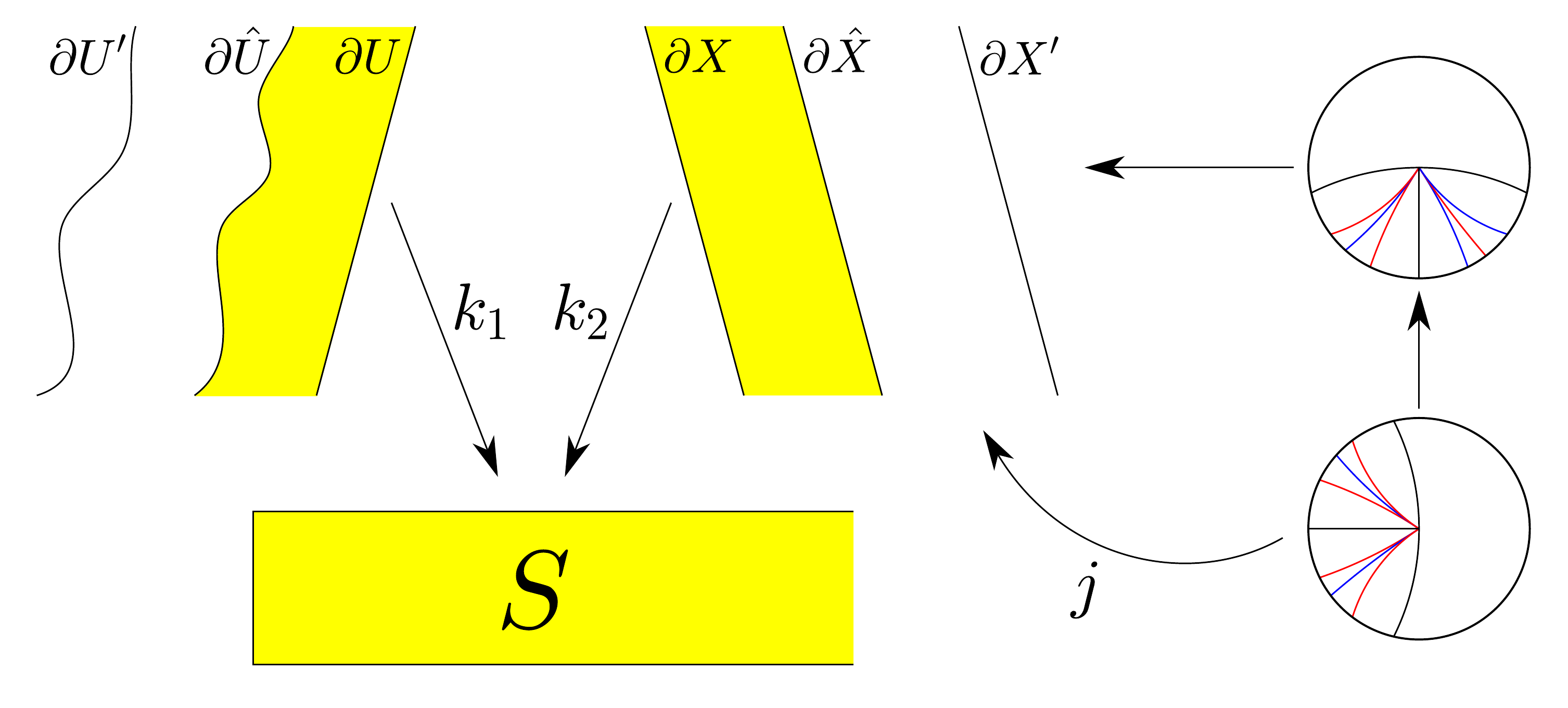}
    \caption{The maps presented in the proof of Lemma \ref{lema.warschawski}. The key for this result is the control of the asymptotics of the maps $k_1$ and $k_2$.}
    \label{fig:warschawski}
\end{figure}

We are finally ready to prove Theorem \ref{main.parabolic_mating}.

\begin{proof}[Proof of Theorem \ref{main.parabolic_mating}]

    Let $\varphi: A\to U\setminus \overline{U'}$ and $\psi: B\to V\setminus \overline{V'}$ be the quasiconformal maps given by Lemma \ref{lema.warschawski}. We begin by constructing the mating at a topological level: let $S$ be the sphere $(U\cup \D\cup V)/\sim$, where the relation $\sim$ identifies any $z \in A$ with $\varphi(z) \in U$ and any $z \in B$ with $\psi(z) \in V$. We then define a correspondence $G$ on $S$ by setting:
    \begin{itemize}
        \item $G: U' \to U$ a $pq:1$ map given by $g\circ f$;
        \item $G: V \to V'$ a $1:pq$ correspondence given by $(f\circ g)^{-1}$;
        \item $G: U' \to \hat{V}$ a $p(q-1):(q-1)$ correspondence given by composing the $p:1$ map $f: U' \to \hat{V}$ with the $q-1:q-1$ correspondence $\Cov_0^g: \hat{V} \to \hat{V}$;
        \item $G: \hat{U} \to V'$ a $(p-1):(p-1)q$ correspondence given by composing the $p-1:p-1$ correspondence $\Cov_0^f: \hat{U} \to \hat{U}$ with the $1:q$ correspondence $g^{-1}: \hat{U} \to V'$;
        \item $G: \hat{U}\setminus \overline{U'} \to \bigcup_n \sigma^n(V)$ a $1:q$ correspondence that maps a point $z$ to $\sigma^n\rho^j(z)$, $1\leq n\leq q$, where $1\leq j\leq p$ is chosen so that $\rho^j(z) \in V$ (equivalently, $\rho^j(z) \in \Delta_\rho$);
        \item $G: \rho^m(U) \to \hat{V}\setminus \overline{U'}$, $1\leq m\leq p$, a $1:1$ correspondence that maps a point $z$ to the $w$ such that $\rho^m\sigma^\ell(w) = z$, where $1\leq \ell\leq q$ is chosen so that $\sigma^\ell(w) \in U$ (equivalently, $\sigma^\ell(w) \in \Delta_\sigma$);
        \item $G$ acts as the group elements $\sigma^n\rho^m$, $1\leq n\leq q$, $1\leq m\leq p$, on the rest of $S$.
    \end{itemize}
    Our construction of $\varphi$ and $\psi$ guarantees that all of these correspondences glue together continuously on boundaries.

    Notice that $K(g\circ f)$ and $K(f\circ g)$ now have the parabolic point at $\infty$ identified in the sphere $S$. The correspondence $G$ induces a decomposition of $S$ into fully invariant sets $\Lambda' := K(g\circ f)\cup K(f\circ g)$ and $\Omega' := \widehat{\C}\setminus \Lambda'$. The action of $G$ on $\Omega'$ is conformally conjugate to that of the group $\Gamma_{p,q}$. Also, $G$ admits pinched polynomial-like restrictions of the correspondences $G|_{K(g\circ f)}$ and $G^{-1}|_{K(f\circ g)}$ that are hybrid equivalent to the pinched polynomial-like restrictions $(g\circ f, U', U)$ and $(f\circ g, V', V)$.

    Let us now define a Beltrami coefficient $\mu$ on the sphere $S$ by:
    \begin{itemize}
        \item $\mu_0$ on $\D$, which can then be viewed as $\varphi^\ast\mu_0$ on $U\setminus \overline{U'}$, and $\psi^\ast\mu_0$ on $V\setminus\overline{V'}$;

        \item $((g\circ f)^n)^\ast\varphi^\ast\mu_0$ on $z$ if $(g\circ f)^n(z) \in U\setminus \overline{U'}$;

        \item $((f\circ g)^n)^\ast\psi^\ast\mu_0$ on $z$ if $(f\circ g)^n(z) \in V\setminus\overline{V'}$;

        \item $\mu_0$ on $K(g\circ f)\cup K(f\circ g)$
    \end{itemize}
    Clearly, $\mu$ is $G$-invariant, and therefore we can find a quasiconformal map $\phi$ integrating $\mu$ such that $F = \phi\circ G\circ \phi^{-1}$ is a holomorphic correspondence. It follow that $F$ is an \textit{algebraic} correspondence, by Chow's Theorem. We see that $F$ is a mating, with $\Lambda_- = \phi(K(g\circ f))$, $\Lambda_+ = \phi(K(f\circ g))$, $\Lambda = \Lambda_-\cup \Lambda_+$, and $\Omega = \widehat{\C}\setminus \Lambda$. Indeed, since $\mu = \mu_0$ on $\D$, the action of $F$ on $\Omega$ is conformally conjugate to that of $G$ on $\Omega'$. Also, $\phi|_{\Lambda}$ has dilatation $0$, and so the pinched polynomial-like restrictions of $F$ induced from $G$ around $\Lambda_-$ and $\Lambda_+$ are hybrid equivalent to the pinched polynomial-like restrictions of $G$ around $\Lambda_-'$ and $\Lambda_+'$, which implies they are hybrid equivalent to the pinched polynomial-like restrictions $(g\circ f, U', U)$ and $(f\circ g, V', V)$.
    
\end{proof}

\begin{figure}
    \centering
    \includegraphics[width=0.45\linewidth]{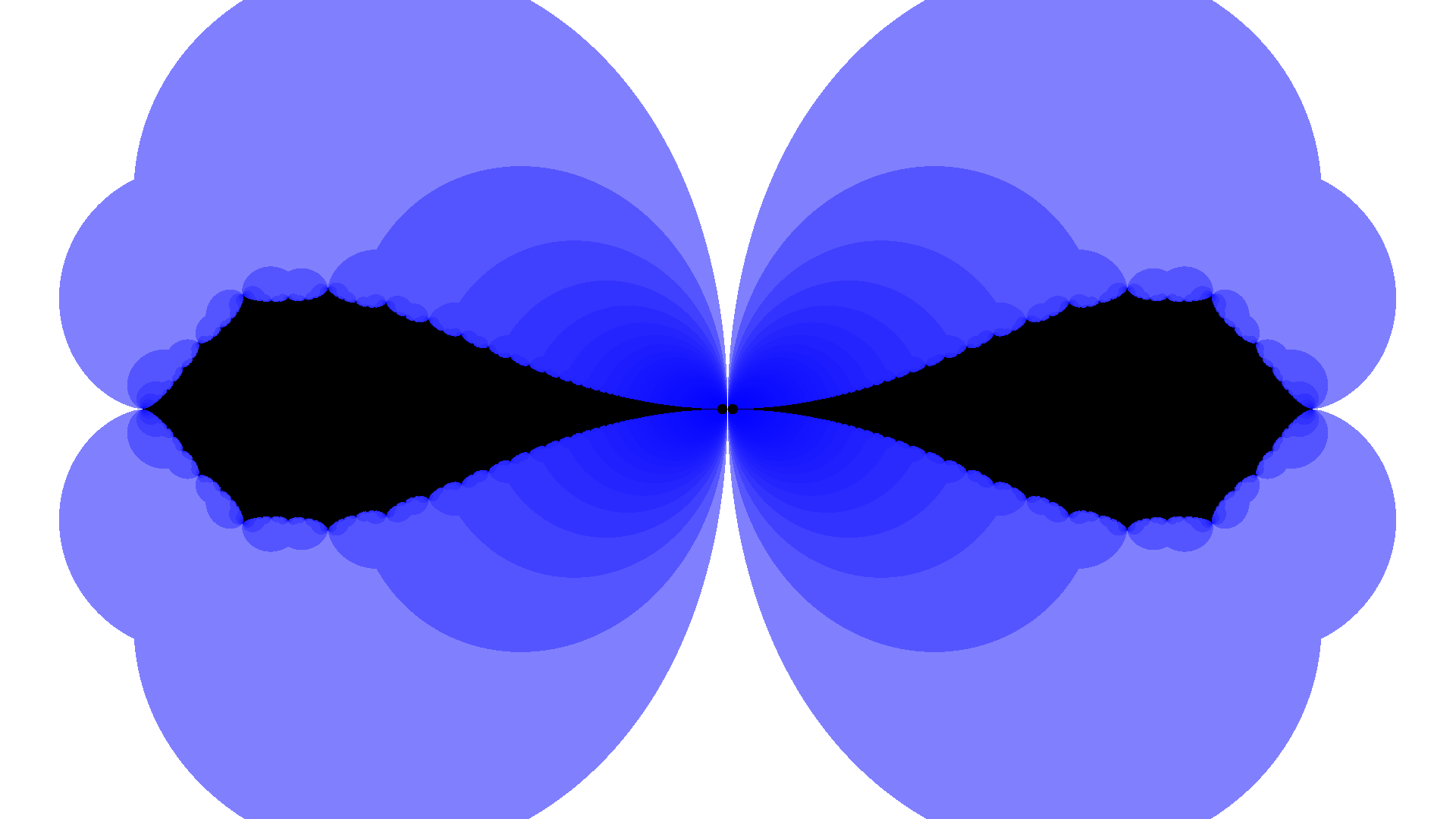}
    \includegraphics[width=0.45\linewidth]{images/c3c3_c=1.111.png}
    \caption{Two examples of matings between pairs of quadratic rational maps and the group $\Gamma_{2,2}$. On the left, $f(z) = g(z) = z^2/(z + 1/2)$, making a symmetric mating. On the right, $f(z) = z^2/(z + 1/2)$ and $g(z) = (2z^2 - z - 1/2)/(2z + 1)$, making an asymmetric mating.}
    \label{fig:matings}
\end{figure}

\begin{rmk}\label{rmk.roots}
    Under the hypothesis we make for the rational maps $f$ and $g$ in Theorem \ref{main.parabolic_mating}, we are excluding a few other rational maps for which similar arguments should allow the construction of matings as well. Namely, if $f$ and $g$ have $\infty$ as a parabolic fixed point with two aligned attracting directions, then the compositions $g\circ f$ and $f\circ g$ have the same attracting directions. If the basins of attraction for each of the directions are connected for the map $g\circ f$, one can easily show that they are as well for $f\circ g$ just like in Proposition \ref{prop.alignment} (since they are images under $f$). Then, one can choose the filled Julia set of $g\circ f$ to be the complement of one of the basins, and the filled Julia set of $f\circ g$ to be its image under $f$. The constructions presented for the proof of Theorem \ref{main.parabolic_mating} all work exactly the same, as the 'external' parabolic direction has been differentiated from the 'internal' one. For instance, take
    \[ f(z) = \frac{(p + 1)z^p + (p - 1)}{(p + 1) + (p - 1)z^p} \ \text{ and } \ g(z) = \frac{(q + 1)z^q + (q - 1)}{(q + 1) + (q - 1)z^q}. \]
    The compositions $g\circ f$ and $f\circ g$ are the same, namely the map
    \[ g\circ f(z) = \frac{(pq + 1)z^{pq} + (pq - 1)}{(pq + 1) + (pq - 1)z^{pq}}. \]
    This can be computed directly, but it is easier to notice that $g\circ f$ and $f\circ g$ both have critical points of order $pq - 1$ at the origin, and the parabolic fixed point at $1$. These maps are the `roots' of the connectedness loci of the families $\Per_1^d(1)$, and represent the external class of the family $\mathcal{B}_d$ in \cite{bullett-lomonaco-lyubich-mukherjee-2024}, for the corresponding degree $d$.
    
    The examples given above are all Blaschke products. In general, Blaschke products which present the same parabolic fixed point fall within the scope of the main result of \cite{luo-mj-mukherjee-2026}, showing where the techniques of both articles intersect.
\end{rmk}

\begin{proof}[Proof of Theorem \ref{main.parabolic_cov}]

    Let $F$ be a mating between the pair $(f, g)$ and the group $\Gamma_{p,q}$, and $\widehat{\C} = \Omega\cup \Lambda$ be the induced decomposition into $F$-invariant sets. Then there are maps $\phi: \Omega \to \D$, $\phi_-: \Lambda_- \to K(g\circ f)$, and $\phi_+: \Lambda_+ \to K(f\circ g)$, such that: $\phi$ is conformal and conjugates the action of $F$ on $\Omega$ with that of the elements $\{\sigma^n\rho^m \ | \ 1\leq m\leq p, 1\leq n\leq q\}$; $\phi_-$ is hybrid (has zero distortion in $\Lambda_-$) and conjugates the map $F|_{\Lambda_-}: \Lambda_- \to \Lambda_-$ to the map $g\circ f: K(g\circ f) \to K(g\circ f)$; and $\phi_+$ is hybrid (has zero distortion in $\Lambda_+$) and conjugates the $(1,pq)$ correspondence $F|_{\Lambda_+}: \Lambda_+ \to \Lambda_+$ to the $(1,pq)$ correspondence $(f\circ g)^{-1}: K(f\circ g) \to K(f\circ g)$.

    We define the `tile' $T := \phi(\Delta_\rho)\cup \Lambda_+$. The action of $\rho$ that identifies the edges of $\Delta_\rho$ induces, via $\phi$, an identification of the edges of $T$, which we denote as $\sim$, so that $\Sigma := T/\sim$ is a sphere. We will define the map $P: \widehat{\C} \to \Sigma$ by setting
    \[ P(z) = \begin{cases} z & \text{if } z \in T, \\ \phi\circ \rho^m\circ \phi^{-1}(z) & \text{if } z \in \phi(\rho^{-m}(\Delta_\rho)), \\ \phi_+^{-1}\circ f\circ \phi_-(z) & \text{if } z \in \Lambda_-. \end{cases} \]
    See Figure $\ref{fig:cov_construction}$. It is clear that $P$ is a quasi-regular map, and furthermore it has zero distortion everywhere since $\phi$ is conformal and $\phi_-, \phi_+$ have zero distortion in $\Lambda_-, \Lambda_+$, respectively. Thus $P$ is a degree $p + 1$ holomorphic map, and the image of the fixed point of $\rho$ under $\phi$ is a branch point of $P$ of order $p$. Thus, for a good choice of coordinates on $\Sigma_1$, $P$ leaves that point completely invariant, meaning it is conjugate to a polynomial.

    Similarly, we can set $S := \phi(\Delta_\sigma)\cup \Lambda_-$, define $\Sigma_2$ to be the sphere obtained by identifying the edges of $S$ via the induced action of $\sigma$, and define the quasi-regular degree $q + 1$ map $Q: \widehat{\C} \to \Sigma_2$ by
    \[ Q(z) = \begin{cases} z & \text{if } z \in S, \\ \phi\circ \sigma^n\circ \phi^{-1}(z) & \text{if } z \in \phi(\sigma^{-n}(\Delta_\sigma)), \\ \phi_-^{-1}\circ g\circ \phi_+(z) & \text{if } z \in \Lambda_+. \end{cases} \]
    Again, $Q$ has zero distortion, and is thus a rational map. The image under $\phi$ of the fixed point of $\sigma$ is a branch point of order $q$ for $Q$, and thus, with a suitable choice of coordinate for $\Sigma_2$, the map $Q$ is conjugate to a polynomial.
    
    Finally, it is clear that $\Cov_0^Q\circ \Cov_0^P$ coincides with $F$ on $\Omega$. Since both $\Cov_0^Q\circ \Cov_0^P$ and $F$ are holomorphic correspondences, it follows (by analytic continuation) that they coincide everywhere. 
    
\end{proof}

\begin{figure}
    \centering
    \includegraphics[width=0.7\linewidth]{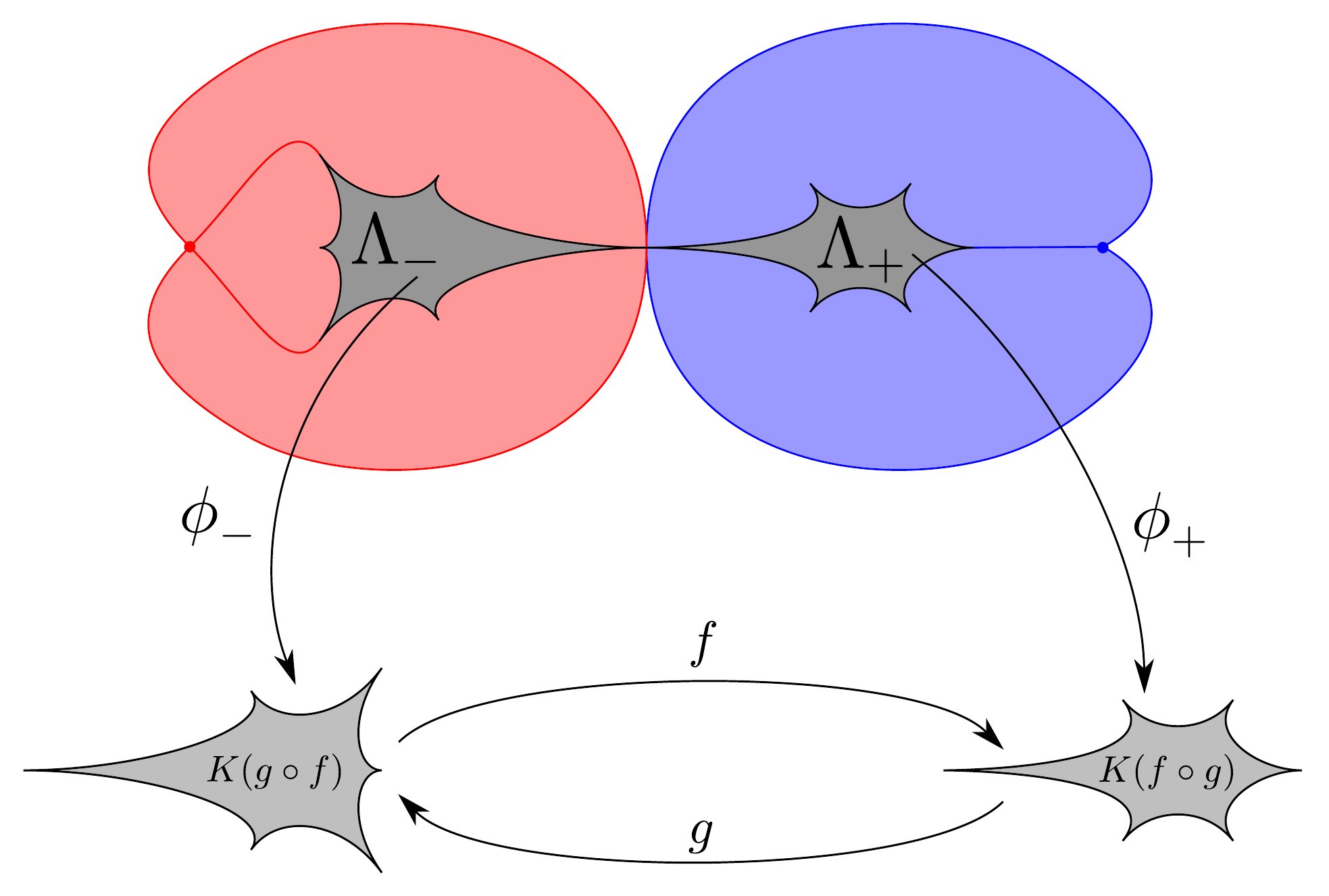}
    \caption{Constructions in the proof of Theorem \ref{main.parabolic_cov}. The tile $T$ is everything outside the red region (i.e. it is the union of the white region, the blue region and $\Lambda_+$). Similarly $S$ is everything outside the blue region. The red and blue dots are the images under $\phi$ of the fixed points of $\rho$ and $\sigma$, respectively. The red region is the union of the sets $\phi(\rho^{-m}(\Delta_\rho))$ for $1\leq m \le p$, and the blue region is the union of the sets $\phi(\sigma^{-n}(\Delta_\sigma))$ for $1\leq n \le q$. The sphere $\Sigma_1$ is the quotient of the tile $T$ resulting from the identification of its upper and lower edges induced by $\rho$. Similarly the sphere $\Sigma_2$ is the quotient of $S$ by the edge-pairing induced by $\sigma$.}
    \label{fig:cov_construction}
\end{figure}

We finish this Section with the proof of Proposition \ref{main.time-reversibility-condition}. Although its proof is quite simple, we again emphasize that the property of time-reversibility was, up until this point, common to all mating phenomena between rational maps and Kleinian groups.

\begin{proof}[Proof of Proposition \ref{main.time-reversibility-condition}]

    Let $\varphi$ be a M\"{o}bius map conjugating $F$ and $F^{-1}$. As $\Omega$ and $\Lambda$ are fully $F$-invariant (and therefore also fully $F^{-1}$-invariant), we immediately see that $\varphi$ preserves both $\Omega$ and $\Lambda$. Around $\Lambda_-$, there is a forward $pq:1$ branch of $F$ that is a pinched polynomial-like map. Thus, $\varphi$ conjugates that with a forward $pq:1$ branch of $F^{-1}$ around some closed subset $\hat{\Lambda} \subset \Lambda$ that is also a pinched polynomial-like. Noticing that $\varphi$ also has to fix the parabolic fixed point lying in the intersection $\Lambda_-\cap \Lambda_+$, and knowing the dynamics of $F^{-1}$ around $\Lambda_-$ is the inverse of a rational map around its filled Julia set, it is clear that $\varphi$ is conjugating $F|_{\Lambda_-}$ with $F^{-1}|_{\Lambda_+}$ conformally around pinched neighborhoods of $\Lambda_-$ and $\Lambda_+$. Since $F|_{\Lambda_-}$ and $F|_{\Lambda_+}$ are hybridly conjugated to $g\circ f|_{K(g\circ f)}$ and $f\circ g|_{K(f\circ g)}$, respectively, we see that $g\circ f$ and $f\circ g$ are hybrid equivalent. Since $K(g\circ f)$ and $K(f\circ g)$ are both connected, this implies that these maps are in fact conformally conjugated.
    
\end{proof}

We return to the example of $f(z) = z^2/(z + 1/2)$ and $g(z) = (2z^2 - z - 1/2)/(2z + 1)$, presented in the right image of Figure \ref{fig:matings}. The map $f$ has its parabolic fixed point at $\infty$, the pre-image of this point at $-1/2$, and a critical point at $0$, which is also fixed. The map $g$ has its parabolic fixed point at $\infty$, the pre-image of this point also at $-1/2$, and a critical point also at $0$, which is mapped to $-1/2$ under $g$. For the composition $g\circ f$, the critical point $0$ is a triple critical point, as $(g\circ f)'(z) = g'(f(z))f'(z)$, with $0$ being a simple zero of $f'(z)$, and a double zero of $g'(f(z))$. On the other hand, $f\circ g$ has $0$ as a simple critical point: $(f\circ g)'(z) = f'(g(z))g'(z)$, with $0$ only being a zero of $g'(z)$, as $g^{-1}(0)$ is comprised of two points, neither of them $0$. Thus, $g\circ f$ and $f\circ g$ cannot be conformally conjugated (they are not even topologically conjugated), meaning that the mating $F$ between the pair $(f, g)$ and $\Gamma_{2,2}$ is not time-reversible.


\section{Hyperbolic surgery}\label{sec.hyperbolic_surgery}

In this section, we will introduce perturbations of $\Gamma_{p,q}$ as a discrete subgroup of $PSL(2,\R)$ (and more generally, of $PSL(2,\C)$), for which all non-elliptic elements are hyperbolic. These groups will be mated with polynomials in the sense of Definition \ref{deft.hyperbolic_mating}. We start with a quick review of some aspects of the iteration theory of polynomials.

\subsection{Polynomials}

The discussion in Section \ref{subsec.parab_maps} becomes more classical in the case of polynomials. If $f$ is a degree $p$ polynomial and $g$ is a degree $q$ polynomial, then $g\circ f$ and $f\circ g$ are both degree $pq$ polynomials. As such, they have associated \textit{filled Julia sets} $K(g\circ f)$ and $K(f\circ g)$, being the complements of the basins of attraction of $\infty$.

\begin{prop}

    $f(K(g\circ f)) = K(f\circ g)$.
    
\end{prop}
\begin{proof}

    Let us denote by $\mathcal{A}(g\circ f)$ and $\mathcal{A}(f\circ g)$ the basins of attraction of $\infty$ of the maps $g\circ f$ and $f\circ g$, respectively. We will show that $f(\mathcal{A}(g\circ f)) = \mathcal{A}(f\circ g)$. Firstly, $z \in \mathcal{A}(g\circ f) \iff (g\circ f)^n(z) \xrightarrow{} \infty$. Thus, the point $f(z)$ satisfies
    \[ (f\circ g)^n(f(z)) = f((g\circ f)^n(z)) \xrightarrow{} \infty, \]
    and therefore $f(\mathcal{A}(g\circ f)) \subseteq \mathcal{A}(f\circ g)$. On the other hand, if $w \in \mathcal{A}(f\circ g)$, then $w = f(z)$ for some $z \in \widehat{\C}$ and
    \[ (g\circ f)^n(z) = g((f\circ g)^{n-1}(f(z))) = g((f\circ g)^{n-1}(w)) \xrightarrow{} \infty \]
    and so $z \in \mathcal{A}(g\circ f)$. Thus $\mathcal{A}(f\circ g) \subseteq f(\mathcal{A}(g\circ f))$, which completes the proof.
    
\end{proof}

As before, we will need to localize the actions of polynomials around their filled Julia sets, and for this we use \textit{polynomial-like maps}, first introduced in \cite{douady-hubbard-1985}. We recall their definition.

\begin{deft}

    A triple $(R, U', U)$ is a \textit{polynomial-like map} if $U, U'$ are simply connected open sets, $U' \Subset U$, and $f: U' \to U$ is a proper holomorphic map.
    
\end{deft}

Given that $g\circ f$ is a polynomial, one may take a simply connected neighborhood $U$ of $K(g\circ f)$ such that $U' := (g\circ f)^{-1}(U)$ is also simply connected and $U' \Subset U$. We notice that $f(U') = g^{-1}(U)$ is also a simply connected neighborhood of $K(f\circ g)$. We can then choose $V$ a larger simply connected neighborhood of $K(f\circ g)$ such that, setting $V' := (f\circ g)^{-1}(V)$, one has $V' \Subset f(U') \Subset V$. See Figure \ref{fig:polynomial_like_restrictions}.

\begin{figure}
    \centering
    \includegraphics[width=0.8\linewidth]{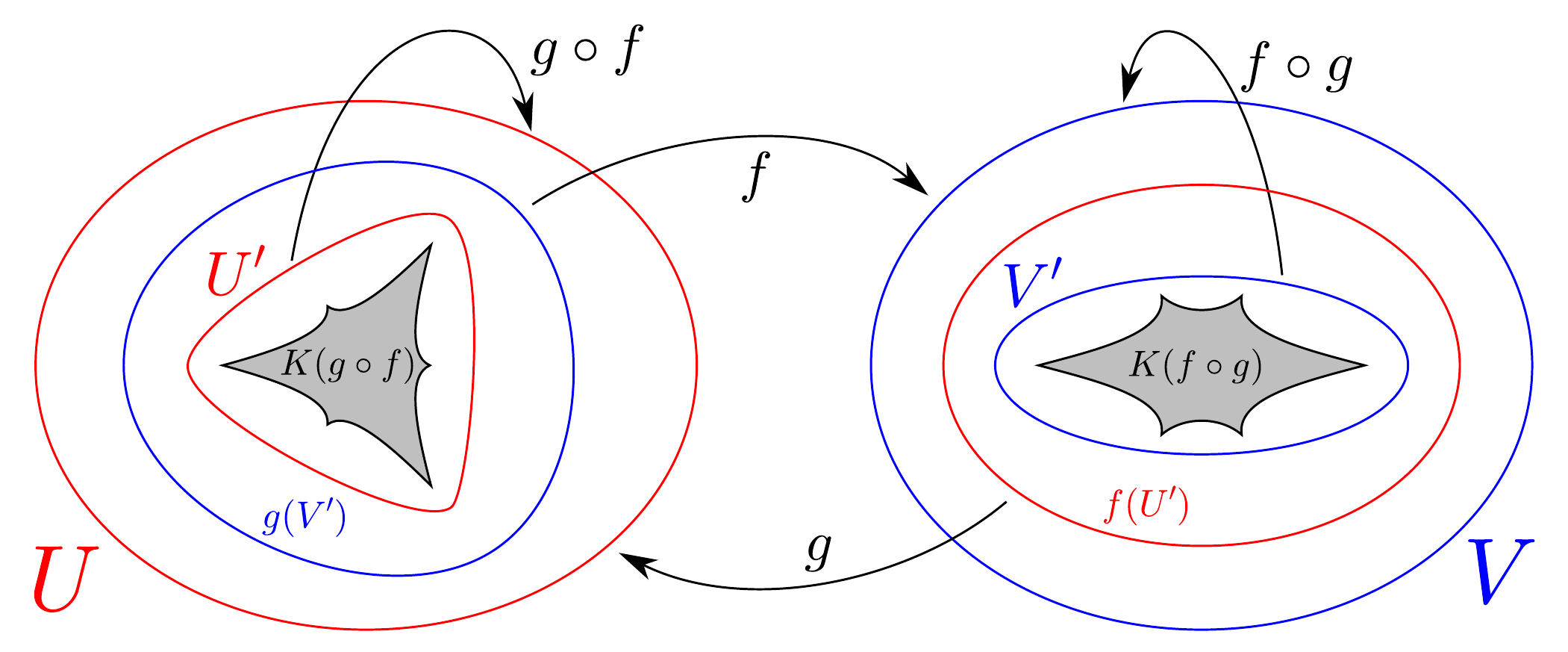}
    \caption{The polynomial-like restrictions of $g\circ f$ and $f\circ g$, and the intermediary domains induced by $f$ and $g$.}
    \label{fig:polynomial_like_restrictions}
\end{figure}

\subsection{Perturbing the group}

We next consider perturbations of $\Gamma_{p,q}$ to representations $\Gamma$ of the group $\C_{p+1}\ast C_{q+1}$ that remain faithful and discrete, but have no parabolic elements. For clarity of exposition we will first consider Fuchsian perturbations (to subgroups of $PSL(2,\R)$). Let $P$ and $Q$ denote the fixed points of $\rho$ and $\sigma$, the generators of $\Gamma_{p,q}$ of order $p+1$ and $q+1$, respectively, and let $\varepsilon$ be real and positive. Let $\varphi_{P, \varepsilon}$ denote the automorphism of the unit disk $\D$ which preserves the real line and exchanges $P$ with $P-\varepsilon$, and let $\varphi_{Q, \varepsilon}$ denote the analogous automorphism which exchanges $Q$ with $Q + \varepsilon$. Then the automorphisms
\[ \sigma_\varepsilon := \varphi_{Q, \varepsilon}\circ \sigma\circ \varphi_{Q, \varepsilon}^{-1} \ \text{ and } \ \rho_\varepsilon := \varphi_{P,\varepsilon}\circ \rho\circ \varphi_{P, \epsilon}^{-1} \]
generate a discrete group $\Gamma_\varepsilon := \left< \sigma_\varepsilon, \rho_\varepsilon \right>$ that is isomorphic to the free product $C_{p+1}\ast C_{q+1}$ of cyclic groups, but the compositions $\rho\sigma$ and $\sigma\rho$ are now hyperbolic. We can find a fundamental domain for $\Gamma_\varepsilon$ in a similar fashion to the way we did for $\Gamma$. We define $\Delta_{\rho_\varepsilon}$ to be the subset of $\D$ bounded by the pair of geodesics symmetrically placed with respect to the real axis, which start at the fixed point $P - \varepsilon$ of $\rho_\varepsilon$, and have angle $2\pi/(p+1)$ between them there. Similarly, we define $\Delta_{\sigma_\varepsilon}$ to be the subset of $\Delta$ bounded by the pair of symmetrically placed geodesics starting at $Q + \varepsilon$ which have angle $2\pi/(q+1)$ between them there.
These fundamental domains $\Delta_{\rho_\varepsilon}$ and $\Delta_{\sigma_\varepsilon}$ for the actions of $\rho_\varepsilon$ and $\sigma_\varepsilon$ respectively, have intersection $\Delta_\varepsilon = \Delta_{\rho_\varepsilon}\cap \Delta_{\sigma_\varepsilon}$ a fundamental domain for $\Gamma_\varepsilon$. See Figure \ref{fig:cp_cq_perturbed_fundamental_domain}.

\begin{figure}
    \centering
    \includegraphics[width=0.4\linewidth]{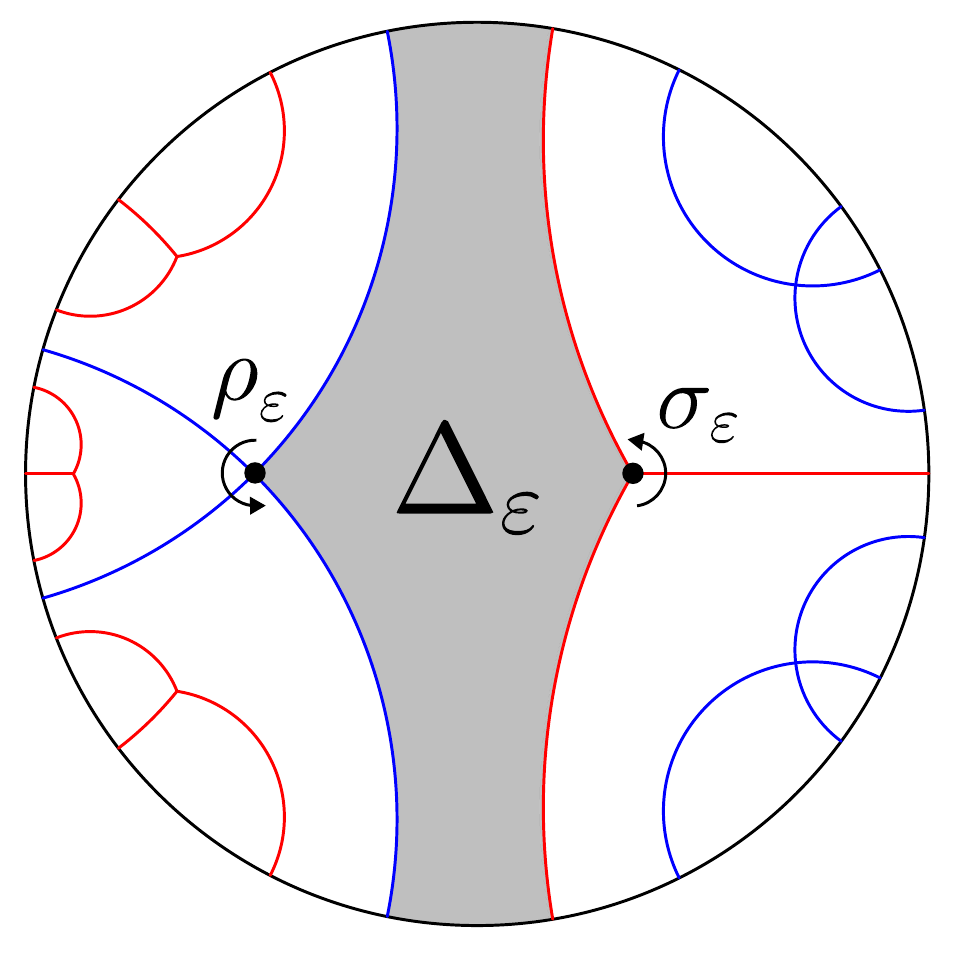}
    \caption{The fundamental domain $\Delta_\varepsilon$ for the action of the perturbed group $\Gamma_{p,q}^\varepsilon$. The first images of the boundary of $\Delta_\varepsilon$ by elements of the group are also depicted.}
    \label{fig:cp_cq_perturbed_fundamental_domain}
\end{figure}

Now regard $\Gamma_\varepsilon$ as a subgroup of $PSL(2,\C)$, acting as isometries of the Poincar\'e $3$-disk and orientation-preserving conformal automorphisms of its boundary, the Riemann sphere $\widehat{\C}$. Consider the action of the involution $\chi(z) = 1/z$. One can see that $\chi\sigma_\varepsilon = \sigma_\varepsilon^{-1}\chi$ and $\chi\rho_\varepsilon = \rho_\varepsilon^{-1}\chi$, since $\chi$ exchanges the fixed points of $\sigma_\varepsilon$ on $\widehat{\C}$ and also exchanges those of $\rho_\varepsilon$. Notice that, since the hemisphere $\D$ is a fundamental domain for $\chi$ on $\widehat{\C}$, the group $G_\varepsilon := \left< \sigma_\varepsilon, \rho_\varepsilon, \chi \right>$ acting on $\widehat{\C}$ has the same regular set as $\Gamma_\varepsilon$ acting on $\widehat{\C}$. Furthermore, it is clear that $\faktor{\D\cup (\widehat{\C}\setminus \overline{\D})}{G_\varepsilon}$ is the same orbifold as $\faktor{\D}{\Gamma_\varepsilon}$.

We return now to the action of $\Gamma_\varepsilon$ on the Poincar\'e $2$-disk $\D$. As in the case of $\Gamma_{p,q}$, we denote the geodesic connecting $i$ to $-i$ by $I$, and, as in that case, we have $I \subset \Delta_\varepsilon$, but now the images of $I$ under the group elements $\sigma_\varepsilon^n\rho_\varepsilon^m$ (respectively $\rho_\varepsilon^m\sigma_\varepsilon^n$), $1\leq n\leq q$, $1\leq m\leq p$, are geodesics on the right (respectively left) side of $I$ which have ends on the boundary of $\D$ that are separated from one another. Nevertheless, we can still consider the domains in $\D$ bounded by $I$ and these sets of images of $I$: let us call them $A_\varepsilon$ and $B_\varepsilon$, respectively. The boundaries of  $A_\varepsilon$ and $B_\varepsilon$ are made up of $I$, its images under the elements $\sigma_\varepsilon^n\rho_\varepsilon^m$ or $\rho_\varepsilon^m\sigma_\varepsilon^n$ of $\Gamma_\varepsilon$ which take $I$ to the relevant side of $I$, and a collection $\mathcal{J}$ of intervals in $\s^1$ between adjacent end points. The two intervals in the boundary of $A_\varepsilon$ that have $i$ and $-i$ as an end point can be identified with each other using the map $\chi$, and similarly the two intervals in the boundary of $B_\varepsilon$ with this property can be identified with one another via $\chi$; making these identifications converts $A_\varepsilon$ and $B_\varepsilon$ into annuli. Our immediate objective now is to find an orientation-reversing element $g_J \in G_\varepsilon$ for every other interval $J \in \mathcal{J}$, such that $g_JJ = J$ (see Figure \ref{fig:perturbed_group_identifications}). If we can find such $g_J$ and we quotient $A_\varepsilon$ and $B_\varepsilon$ by these identifications in addition to our initial identifications by $\chi$, we will create annuli $A$ and $B$, respectively, where the inner boundary of $A$ is made up of the images of $I$ under $\rho_\varepsilon^m\sigma_\varepsilon^n$, and the inner boundary of $B$ is made up of the images of $I$ under $\sigma_\varepsilon^n\rho_\varepsilon^m$. Then the elements $\sigma_\varepsilon^n\rho_\varepsilon^m$ and $\rho_\varepsilon^m\sigma_\varepsilon^n$ of $\Gamma_\varepsilon$ will map the inner boundaries of these annuli to outer boundaries as $pq:1$ maps, in exactly the same way that $g\circ f$ and $f\circ g$ map the boundaries of the domains of their polynomial-like restrictions (in Figure \ref{fig:polynomial_like_restrictions}).

\begin{lema}

    Let $J$ be an interval in $\mathcal{J}$ that does not have $i$ or $-i$ as an end point. If $J$ contains an endpoint of one of the geodesic rays that bound $\Delta_{\rho_\varepsilon}$ or of one of its images under $\rho_\varepsilon$, then $J = \chi\rho_\varepsilon^m J$ for some $1\leq m\leq p$. If $J$ contains an endpoint of one of the geodesic rays that bound $\Delta_{\sigma_\varepsilon}$ or of one of its images under $\sigma_\varepsilon$, then $J = \chi\sigma_\varepsilon^n J$ for some $1\leq n\leq q$. Otherwise, there exists an element $g = \rho_\varepsilon^{n_3}\sigma_\varepsilon^{n_2}\rho_\varepsilon^{n_1}$ or $\sigma_\varepsilon^{n_3}\rho_\varepsilon^{n_2}\sigma_\varepsilon^{n_1}$ such that $J = \chi gJ$.
    
\end{lema}
\begin{proof}

    We first notice that the endpoints of the intervals in $\mathcal{J}$ are images of $\pm i$ under elements $\sigma_\varepsilon^n\rho_\varepsilon^m$ and $\rho_\varepsilon^m\sigma_\varepsilon^n$ of $\Gamma_\varepsilon$. Since $\{i, -i\}$ is a $\chi$-invariant set, and $\chi$ anti-commutes with the generators $\sigma_\varepsilon$ and $\rho_\varepsilon$ of $\Gamma_\varepsilon$, their images also form a $\chi$-invariant set. Similarly, the set of endpoints of the geodesic rays bounding $\Delta_{\rho_\varepsilon}$ is $\chi$-invariant (because these geodesic rays are taken symmetric about the real axis), and thus the set of their images under $\rho_\varepsilon$ is also $\chi$-invariant; let us call this set $C_b$. An analogous claim clearly holds for the endpoints of the geodesic rays bounding $\Delta_{\sigma_\varepsilon}$ and their images under $\sigma_\varepsilon$; the $\chi$-invariant set in this case will be denoted $C_r$.
    
    The first case in the statement then follows immediately: $\chi$ maps every point of $C_b$ to a point of $C_b$ and the endpoints of intervals in $\mathcal{J}$ to endpoints of intervals in $\mathcal{J}$; if $J$ contains a point $x \in C_b$, then $\chi$ maps $J$ to some other $J' \in \mathcal{J}$ containing $\chi(x) \in C_b$. Now, intervals containing points of $C_b$ all lie on the boundary of $\Delta_{\sigma_\varepsilon}$, and thus their images under powers of $\rho_\varepsilon$ that do not send them to $\Delta_{\rho_\varepsilon}$ are also in the collection $\mathcal{J}$. In particular, if we take $1\leq m\leq p$ such that $\chi(x) = \rho_\varepsilon^m(x)$, then we find that $\chi I = \rho_\varepsilon^m I$, or equivalently that $I = \chi\rho_\varepsilon^m I$. A completely analogous argument also gives us the second statement.
    
    Let us now consider the final case. Assume that $J$ lies on the boundary of $\Delta_{\sigma_\varepsilon}$ --- the case when it lies on the boundary of $\Delta_{\rho_\varepsilon}$ is completely analogous. Then $\chi J \in \mathcal{J}$ also lies on the same boundary, and there must exist integers $m_1, m_2$ such that $J \subset \rho_\varepsilon^{m_1}\Delta_{\rho_\varepsilon}$ and $\chi J \subset \rho_\varepsilon^{m_2}\Delta_{\rho_\varepsilon}$. Thus the intervals $\rho_\varepsilon^{-m_1}J$ and $\rho_\varepsilon^{-m_2}\chi J$ are in $D_{\rho_\varepsilon}$, and each contains a point of $C_r$. Notice that these intervals are not in the collection $\mathcal{J}$. Nevertheless, there exists some power $m_3$ of $\sigma_\varepsilon$ such that $\sigma_\varepsilon^{m_3}\rho_\varepsilon^{-m_1}I = \rho_\varepsilon^{-m_2}\chi I$, which yields the statement claimed, with $n_1 = -m_1, n_2 = m_3$ and $n_3 = m_2$.
    
\end{proof}

\begin{figure}
    \centering
    \includegraphics[width=0.7\linewidth]{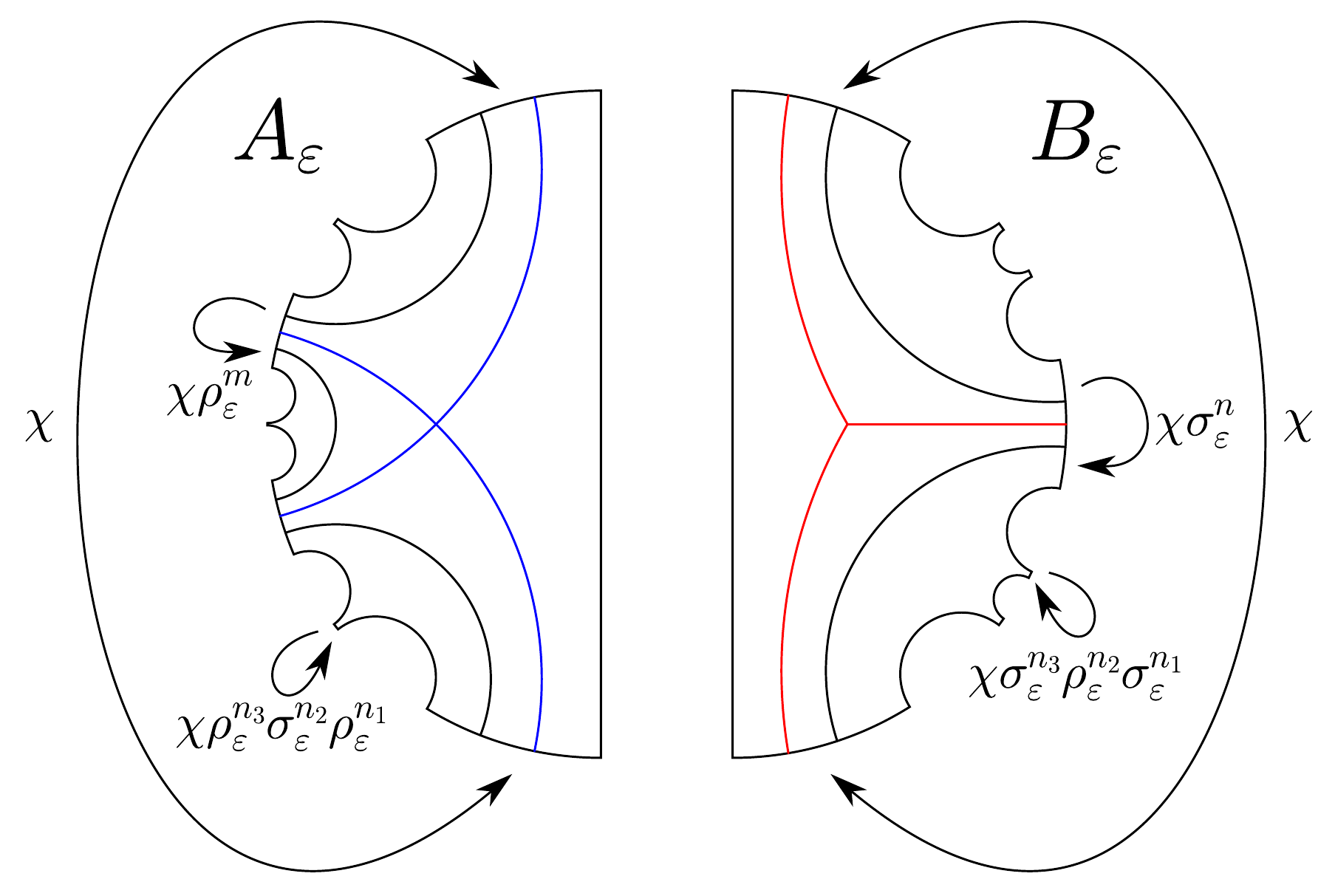}
    \caption{An example of the identifications that convert each set $A_\varepsilon$ and $B_\varepsilon$ into an annulus. The endpoints of the curves in blue are the set $C_b$, while the endpoints of the curves in red are the set $C_r$.}
    \label{fig:perturbed_group_identifications}
\end{figure}

We now observe that the elements $g_J$ found by the previous lemma are compositions of powers of $\sigma_\varepsilon$ and $\rho_\varepsilon$ with a single $\chi$. Since $\sigma_\varepsilon$ and $\rho_\varepsilon$ are orientation-preserving on $\s^1$, and $\chi$ is orientation-reversing, $g_J$ has the  property we were seeking. We denote by $A$ and $B$ the annuli obtained from making all the identifications on $A_\varepsilon$ and $B_\varepsilon$.

We note that now the intermediary curves $\rho_\varepsilon^m I$ and $\sigma_\varepsilon^n I$, $1\leq m\leq p, 1\leq n\leq q$, will have their end points identified with one another, since they also form a $\chi$-invariant set. Thus, in $A$ and $B$, they descend to closed loops that divide each of these annuli into two concentric annuli; see Figure \ref{fig:perturbed_grouop_polynomial_like}. We shall denote by $\hat{A}$ and $\hat{B}$ the innermost of these narrower annuli.

\begin{figure}
    \centering
    \includegraphics[width=0.7\linewidth]{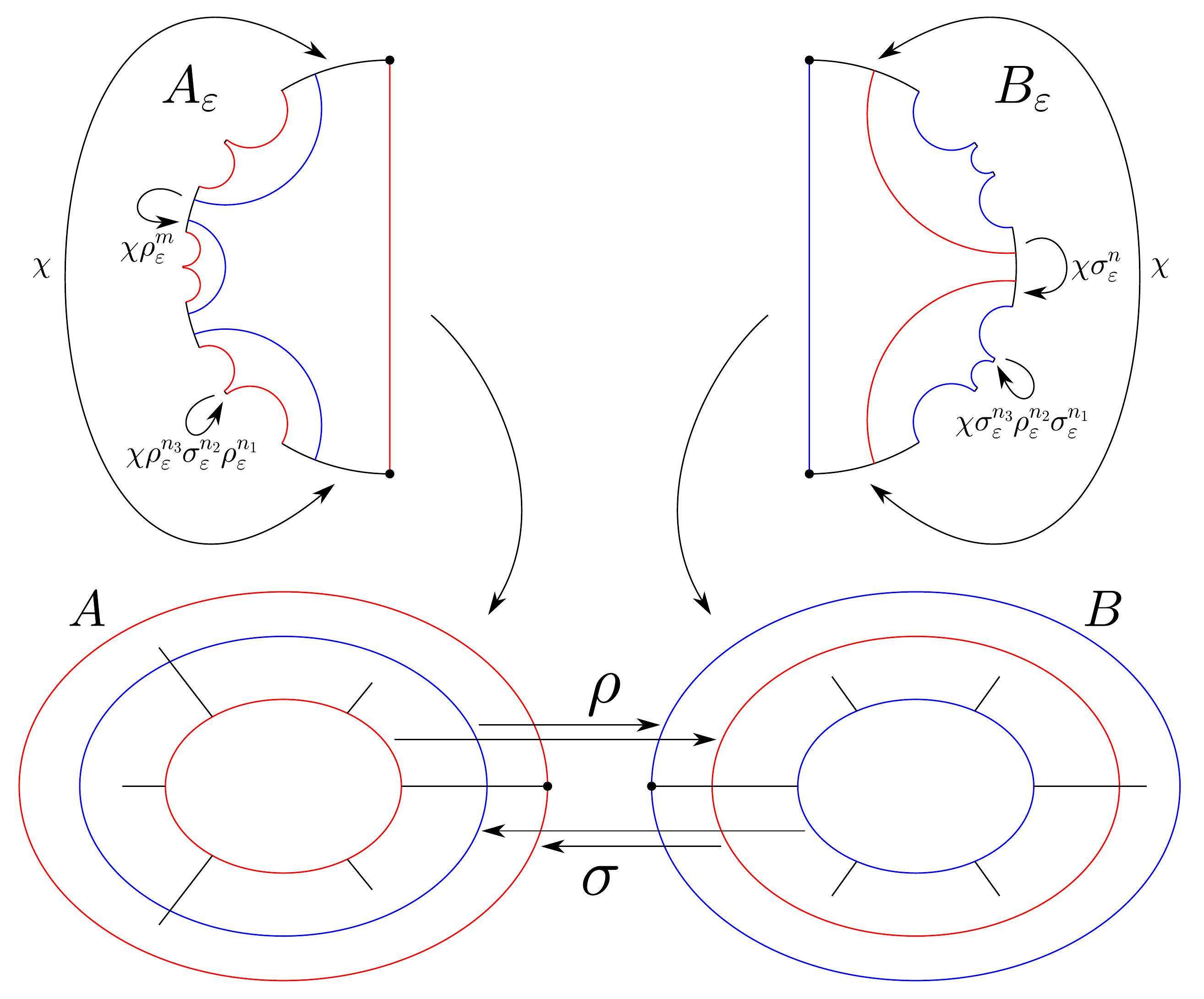}
    \caption{The annuli $A$ and $B$, and the induced actions of $\rho$ and $\sigma$ on their boundaries and intermediary curves. The precise power of $\rho$ or $\sigma$ varies along these curves. Notice the similarity with Figure \ref{fig:polynomial_like_restrictions}.}
    \label{fig:perturbed_grouop_polynomial_like}
\end{figure}

\subsection{Proofs of Theorems \ref{main.hyperbolic_mating} and \ref{main.hyperbolic_cov}}

We first remark that since the  faithful discrete representations of $C_{p+1} \ast C_{q+1}$ in $PSL(2,\C)$ constitute a single quasi-conformal conjugacy class, we lose no generality by working with $\Gamma_\varepsilon\subset PSL(2,\R)$ as a representative of this class of subgroups of $PSL(2,\C)$.

The definition of orbit mating used for Theorems \ref{main.hyperbolic_mating} and \ref{main.hyperbolic_cov} follows that in \cite{bullett-harvey-2000, ratis-laude-2025}: note that it differs from the usual definition used in the parabolic case (Definition \ref{deft.parabolic_mating}), in that on the group side we now only ask for grand orbit equivalence between the correspondence $F$ restricted to $\Omega$ and the group $\Gamma$ on $\D$ (see also Remark \ref{rmk.definition_by_quotient}).

\begin{deft}\label{deft.hyperbolic_mating}

    Let $f, g$ be polynomials of degrees $p, q$, respectively, and let $\Gamma$ be quasiconformally conjugate to $\Gamma_{p,q}^\varepsilon$. We say that a $pq:pq$ correspondence $F$ is an \textit{orbit mating} between the pair $(f, g)$ and the group $\Gamma$ if there exists a decomposition $\widehat{\C} = \Lambda\cup \Omega$ into two disjoint fully $F$-invariant sets such that:
    \begin{itemize}
        \item $\Lambda = \Lambda_-\cup \Lambda_+$ is the union of two disjoint compact sets, such that the $pq:1$ restriction $F|_{\Lambda_-}:\Lambda_- \to \Lambda_-$ admits a polynomial-like restriction hybrid conjugate to $g\circ f$, and the $pq:1$ restriction $F^{-1}|_{\Lambda_+}:\Lambda_+ \to \Lambda_+$ is hybrid conjugate to $f\circ g$;

        \item $\Omega$ is an open set and $\faktor{\Omega}{F}$ is a Riemann surface biholomorphic to $\faktor{\D}{\Gamma}$.
    \end{itemize}
    
\end{deft}

As in the proof of Theorem \ref{main.parabolic_mating}, we must find quasiconformal maps $\varphi: A \to U\setminus \overline{U'}$ and $\psi: B \to V\setminus \overline{V'}$ such that the following diagrams are commutative:
\begin{equation}\label{eq.commutative_diagrams_2} \begin{tikzcd}
\partial_i A \arrow[r, "\{\rho^m\}_m"] \arrow[d, "\varphi"'] & \partial_i\hat{B} \arrow[d, "\psi"] \arrow[r, "\{\sigma^n\}_n"] & \partial_o A \arrow[d, "\varphi"] & \partial_i B \arrow[d, "\psi"'] \arrow[r, "\{\sigma^n\}_n"] & \partial_i \hat{A} \arrow[r, "\{\rho^m\}_m"] \arrow[d, "\varphi"] & \partial_o B \arrow[d, "\psi"] \\
\partial U' \arrow[r, "f"']                                  & \partial\hat{V} \arrow[r, "g"']                                       & \partial U                  & \partial V' \arrow[r, "g"']                                       & \partial \hat{U} \arrow[r, "f"']                                  & \partial V            
\end{tikzcd} \end{equation}
where $\partial_o$ denotes the outer boundary of an annulus, while $\partial_i$ denotes the inner boundary. This time, however, we only need the classical tools of the theory of quasiconformal maps.

\begin{lema}\label{lema.hyperbolic_interpolations}

    For any choice of diffeomorphisms $\varphi: \partial_o A \to \partial U$ and $\psi: \partial_o B \to \partial V$, we can find quasiconformal extensions $\varphi: A\to U\setminus \overline{U'}$ and $\psi: B\to V\setminus \overline{V'}$ such that the diagrams (\ref{eq.commutative_diagrams_2}) are commutative.
    
\end{lema}
\begin{proof}

    For the same reason as in the proof of Lemma \ref{lema.diagram_lift}, we may lift any chosen diffeomorphisms $\varphi: \partial_o A \to \partial U$ and $\psi: \partial_o B \to \partial V$ to diffeomorphisms $\varphi: \partial_i A\cup \partial_i \hat{A} \to \partial U'\cup \partial \hat{U}$ and $\psi: \partial_i B\cup \partial_i \hat{B} \to \partial V'\cup \partial \hat{V}$. These maps are quasisymmetric, and therefore admit quasiconformal extensions to the annuli $\hat{A}, A\setminus \hat{A}, \hat{B}$, and $B\setminus \hat{B}$. Gluing these extensions together, we find the desired maps.
    
\end{proof}

We are now ready to prove Theorem \ref{main.hyperbolic_mating}.

\begin{proof}[Proof of Theorem \ref{main.hyperbolic_mating}]

    Take $\varphi: A\to U\setminus \overline{U'}$ and $\psi: B\to V\setminus \overline{V'}$ as in Lemma \ref{lema.hyperbolic_interpolations}. We again begin by constructing the mating at a topological level: we let $S$ be the sphere $(U\cup V)/\sim$, where the relation $\sim$ identifies a point $z \in \partial U$ with $w \in \partial V$ if $\varphi^{-1}(z) = \psi^{-1}(w) \in \partial_o A = \partial_o B$. We then define a correspondence $G$ on $S$ by setting:
    \begin{itemize}
        \item $G: U' \to U$ a $pq:1$ map given by $g\circ f$;
        \item $G: V \to V'$ a $1:pq$ correspondence given by $(f\circ g)^{-1}$;
        \item $G: U' \to \hat{V}$ a $p(q-1):(q-1)$ correspondence given by composing the $p:1$ map $f: U' \to \hat{V}$ with the $q-1:q-1$ correspondence $\Cov_0^g: \hat{V} \to \hat{V}$;
        \item $G: \hat{U} \to V'$ a $(p-1):(p-1)q$ correspondence given by composing the $p-1:p-1$ correspondence $\Cov_0^f: \hat{U} \to \hat{U}$ with the $1:q$ correspondence $g^{-1}: \hat{U} \to V'$;
        \item $G: \hat{U}\setminus \overline{U'} \to \bigcup_n \sigma^n(V)$ a $1:q$ correspondence that maps a point $z$ to $\sigma^n\rho^j(z)$, $1\leq n\leq q$, where $1\leq j\leq p$ is chosen so that $\rho^j(z) \in V$ (equivalently, $\rho^j(z) \in \Delta_\rho$);
        \item $G: \rho^m(U) \to \hat{V}\setminus \overline{U'}$, $1\leq m\leq p$, a $1:1$ correspondence that maps a point $z$ to the point $w$ such that $\rho^m\sigma^\ell(w) = z$, where $1\leq \ell\leq q$ is chosen so that $\sigma^\ell(w) \in U$ (equivalently, $\sigma^\ell(w) \in \Delta_\sigma$);
        \item $G$ acts as the group elements $\sigma_\varepsilon^n\rho_\varepsilon^m$, $1\leq n\leq q$, $1\leq m\leq p$, on the rest of $S$.
    \end{itemize}
    Our construction of $\varphi$ and $\psi$ gives us that the various pieces match on boundaries, ensuring $G$ is well-defined and continuous.

    Just as in the parabolic case, we have a decomposition of $S$ into fully $G$-invariant sets $\Lambda' := K(g\circ f)\cup K(f\circ g)$ and $\Omega' := \widehat{\C}\setminus \Lambda'$. We have that $G$ admits polynomial-like restrictions of the actions of $G|_{K(g\circ f)}$ and $G^{-1}|_{K(f\circ g)}$ that are hybrid equivalent to the polynomial-like restrictions $(g\circ f, U', U)$ and $(f\circ g, V', V)$, but now the action of $G$ on $\Omega'$ cannot be conjugate to that of $\Gamma_\varepsilon$ on $\D$ (for the elementary reason that $\Omega'$ is an annulus). Nevertheless, it remains true that $\faktor{\Omega'}{G}$ is biholomorphic to $\faktor{A_\varepsilon\cup B_\varepsilon}{G_\varepsilon} = \faktor{\D}{\Gamma_\varepsilon}$.

    Now let $\phi$ be the quasiconformal map conjugating $\Gamma$ to $\Gamma_\varepsilon$. We define a Beltrami form $\mu$ in the following way:
    \begin{itemize}
        \item $\mu = \phi^\ast\mu_0$ on $A\cup B$, which can then be seen as $\varphi^\ast\phi^\ast\mu_0$ on $U\setminus \overline{U'}$, and $\psi^\ast\phi^\ast\mu_0$ on $V\setminus\overline{V'}$;

        \item $((g\circ f)^n)^\ast\varphi^\ast\phi^\ast\mu_0$ on $z$ if $(g\circ f)^n(z) \in U\setminus \overline{U'}$;

        \item $((f\circ g)^n)^\ast\psi^\ast\phi^\ast\mu_0$ on $z$ if $(f\circ g)^n(z) \in V\setminus\overline{V'}$;

        \item $\mu_0$ on $K(g\circ f)\cup K(f\circ g)$
    \end{itemize}
    Clearly, $\mu$ is $G$-invariant, and therefore we can find a quasiconformal map $\Phi$ integrating $\mu$ such that $F = \Phi\circ G\circ \Phi^{-1}$ is an algebraic correspondence. We see that $F$ is an orbit mating, with $\Lambda_- = \Phi(K(g\circ f))$, $\Lambda_+ = \Phi(K(f\circ g))$, $\Lambda = \Lambda_-\cup \Lambda_+$, and $\Omega = \widehat{\C}\setminus \Lambda$. Indeed, since $\Phi|_{\Lambda}$ has dilatation $0$, the polynomial-like restrictions of $F$ induced from $G$ around $\Lambda_-$ and $\Lambda_+$ are hybrid equivalent to the polynomial-like restrictions of $G$ around $\Lambda_-'$ and $\Lambda_+'$, which implies they are hybrid equivalent to the polynomial-like restrictions $(g\circ f, U', U)$ and $(f\circ g, V', V)$. Also, $\mu = \Phi^\ast\mu_0$ means that $\faktor{\Omega}{F}$ is quasiconformally conjugate to $\faktor{\Omega'}{G}$ via a map induced by $\Phi$. But $\Phi^\ast\mu_0$ on $A\cup B$ coincides with $\phi^\ast\mu_0$, which means that $\faktor{\Omega}{F}$ is biholomorphic to $\faktor{\D}{\Gamma}$.
    
\end{proof}

The proof of Theorem \ref{main.hyperbolic_cov} is very similar to that of Theorem \ref{main.parabolic_cov}, but now we rely on the way the mating has been constructed, since there is no map conjugating the action of the correspondence $F$ on $\Omega$ to that of the group $\Gamma$ on $\D$.

\begin{proof}[Proof of Theorem \ref{main.hyperbolic_cov}]

    Given the orbit mating $F$ between the pair $(f, g)$ and the group $\Gamma$ constructed in the proof of Theorem \ref{main.hyperbolic_mating}, we notice that, by construction, there exists a closed curve $\gamma$ and two polynomial-like restrictions $(F|_{U'}, U', U)$ and $(F^{-1}|_{V'}, V', V)$ with limit sets $\Lambda_-$ and $\Lambda_+$, respectively, satisfying $\partial U = \partial V = \gamma$, such that there are two quasiconformal maps $\phi_-: A\to  U\setminus \overline{U'}$ and $\phi_+: B\to V\setminus \overline{V'}$ conjugating the actions on the boundaries. Now let $T := \phi_-(A_\varepsilon\cap \Delta_{\rho_\varepsilon})$. Notice that the map $f$ induces a map $\tilde{f}$ from the union of $U\setminus T$ with $\bigcup_{m=1}^q(\phi_-(\rho_\varepsilon^m(A_\varepsilon\cap \Delta_{\rho_\varepsilon})))$, to $V$. We set $\Sigma_1 := (T\cup V)/\sim$, where $\sim$ identifies two edges of $T$ via $\phi_-\circ \rho_\varepsilon\circ \phi_-^{-1}$ and observe that $\Sigma_1$ is a sphere. We define a branched-covering map $P: \widehat{\C} \to \Sigma_1$ by setting:
    \[ P(z) = \begin{cases} z & \text{if } z \in T\cup V \\ \phi_-\circ \rho_\varepsilon^m\circ \phi_-^{-1}(z) & \text{if } z \in \phi_-(\rho^{-m}(A_\varepsilon\cap \Delta_{\rho_\varepsilon})) \\ \tilde{f}(z) & \text{otherwise.} \end{cases} \]
    Notice that $P$ is holomorphic on the sets $\phi_-(\rho^{-m}(A_\varepsilon\cap \Delta_{\rho_\varepsilon}))$ since $\phi_-$ conjugates the action of $\rho_\varepsilon$ to a holomorphic map on $\widehat{\C}$. It is clear that near the image under $\phi_-$ of the fixed point of $\rho_\varepsilon$, the map $P$ behaves like $z \to z^{p+1}$, and that $\deg P = p + 1$, so that with a suitable choice of coordinate on $\Sigma_1$, it becomes a rational map conjugate to a polynomial of degree $p+1$. Similarly, we can set $S := \phi_+(B_\varepsilon\cap \Delta_{\sigma_\varepsilon})$, denote by $\tilde{g}$ the map $g$ induces from the union of $V\setminus S$ and $\bigcup_n(\phi_+(\sigma_\varepsilon^n(B_\varepsilon\cap \Delta_{\sigma_\varepsilon})))$ to $U$, construct a sphere $\Sigma_2 := (S\cup U)/\sim$, and define $Q: \widehat{C} \to \Sigma_2$ by
    \[ Q(z) = \begin{cases} z & \text{if } z \in S\cup U \\ \phi_+\circ \sigma^n\circ \phi_+^{-1}(z) & \text{if } z \in \phi_+(\sigma_\varepsilon^{-n}(B_\varepsilon\cap \Delta_{\sigma_\varepsilon})) \\ \tilde{g}(z) & \text{otherwise}. \end{cases} \]
    Again, $Q$ is conjugate to a polynomial, its `infinity' now at the image under $\phi_+$ of the fixed point of $\sigma_\varepsilon$. Finally, it is clear that $\Cov_0^Q\circ \Cov_0^P$ coincides with $F$ on $T\cup S$, and therefore, since both are holomorphic correspondences, they coincide everywhere.
    
\end{proof}

\bibliographystyle{plain}
\bibliography{biblio}

\end{document}